\documentclass{amsart}

\usepackage{amssymb,amscd,amsthm,amsxtra}
\usepackage{latexsym}

\usepackage{color}
\usepackage[usenames,dvipsnames,svgnames,table]{xcolor}
\bibliography{refs}

\usepackage{hyperref}
\hypersetup{colorlinks = false}

\usepackage[mathscr]{euscript}
\renewcommand{\mathcal}[1]{{\mathscr#1}}
\newcommand{\CRES}{2^*_s}

\newtheorem{theorem}{Theorem}[section]

\newtheorem{lemma}[theorem]{Lemma}
\newtheorem{prop}[theorem]{Proposition}

\theoremstyle{definition}
\newtheorem{defn}[theorem]{Definition}
\theoremstyle{remark}

\numberwithin{equation}{section}

\newcommand{\R}{{\mathbb R}}
\newcommand{\N}{{\mathbb N}}

\renewcommand{\leq}{\leqslant}
\renewcommand{\le}{\leqslant}
\renewcommand{\geq}{\geqslant}
\renewcommand{\ge}{\geqslant}

\newcommand{\eps}{\varepsilon }
\renewcommand{\epsilon}{\varepsilon }

\newlength{\defbaselineskip}
\newcommand{\setlinespacing}[1]
           {\setlength{\baselineskip}{#1 \defbaselineskip}}

\author[S. Dipierro]{Serena Dipierro}
\address[Serena Dipierro]{School of Mathematics and Statistics,
University of Melbourne, 813 Swanston St, Parkville VIC 3010, Australia,
School of Mathematics and Statistics,
University of Western Australia,
35 Stirling Highway,
Crawley, Perth
WA 6009, Australia,
and
Weierstra{\ss}-Institut f\"ur Angewandte
Analysis und Stochastik, Hausvogteiplatz 5/7, 10117 Berlin, Germany.}
\email{serena.dipierro@ed.ac.uk}

\author[M. Medina]{Mar\'{i}a Medina}
\address[Mar\'{i}a Medina]{Departamento de Matem\'aticas, Universidad Aut\'onoma de Madrid,
        28049, Madrid, Spain, and
Departamento de Ingenieria Matem\'atica, 
Facultad de Ciencias Fisicas y Matem\'aticas, 
Universidad de Chile,
Casilla 170, Correo 3, Santiago, Chile. }
\email{maria.medina@uam.es}

\author[I. Peral]{Ireneo Peral}
\address[Ireneo Peral]{Departamento de Matem\'aticas, Universidad Aut\'onoma de Madrid,
        28049, Madrid, Spain.}
\email{ireneo.peral@uam.es}

\author[E. Valdinoci]{Enrico Valdinoci}
\address[Enrico Valdinoci]{School of Mathematics and Statistics,
University of Melbourne, 813 Swanston St, Parkville VIC 3010, Australia,
School of Mathematics and Statistics,
University of Western Australia,
35 Stirling Highway,
Crawley, Perth
WA 6009, Australia,
Weierstra{\ss}-Institut f\"ur Angewandte
Analysis und Stochastik, Hausvogteiplatz 5/7, 10117 Berlin, Germany,
Dipartimento di Matematica, Universit\`a degli studi di Milano,
Via Saldini 50, 20133 Milan, Italy, and
Istituto di Matematica Applicata e Tecnologie Informatiche,
Consiglio Nazionale delle Ricerche,
Via Ferrata 1, 27100 Pavia, Italy.}
\email{enrico@math.utexas.edu}

\begin{document}

\subjclass[2010]{35B40, 35D30, 35J20, 35R11, 49N60}

\keywords{Bifurcation, Lyapunov-Schmidt reduction, critical problem, fractional elliptic regularity.}

\thanks{{\it Acknowledgements}.
The first author has been supported by Alexander von Humboldt Foundation and EPSRC grant  EP/K024566/1
\emph{Monotonicity formula methods for nonlinear PDEs}.
The second and third authors have been supported by projects  MTM2010-18128 and MTM2013-40846-P, MINECO, Spain.
The fourth author has been supported by ERC grant 277749 \emph{EPSILON Elliptic
Pde's and Symmetry of Interfaces and Layers for Odd Nonlinearities}. We would like to thank the Referee for her or his very
accurate and very valuable job.}

\title[Bifurcation results for a fractional equation]
{Bifurcation results for a fractional elliptic equation with critical
exponent in $\mathbb{R}^n$}

\begin{abstract}
In this paper we study some nonlinear elliptic equations in~$\R^n$
obtained as a perturbation of the problem with the fractional critical Sobolev
exponent, that is
\begin{equation*}
(-\Delta)^s u = \epsilon\,h\,u_+^q + u_+^p \ {\mbox{ in }}\R^n,
\end{equation*}
where
$s\in(0,1)$, $n>4s$, $\epsilon>0$ is a 
small parameter, $p=\frac{n+2s}{n-2s}$, $0<q<p$
and~$h$ is a continuous and compactly supported function.

To construct solutions to
this equation, we use the Lyapunov-Schmidt reduction,
that takes advantage of the variational structure of the problem.
For this, the case $0<q<1$ is particularly difficult, due to the lack
of regularity of the associated energy functional,
and we need to introduce a new functional setting
and develop an appropriate fractional elliptic regularity theory.
\end{abstract}

\maketitle

{\small
\tableofcontents
}

\section{Introduction}

In this paper we deal with the problem
\begin{equation}\label{problem}
(-\Delta)^s u = \epsilon\,h\,u_+^q + u_+^p \ {\mbox{ in }}\R^n,
\end{equation}
where $s\in(0,1)$ and $(-\Delta)^s$ is the fractional Laplacian, that is
$$ (-\Delta)^s u(x)= c_{n,s}\,PV \int_{\R^n}\frac{u(x)-u(y)}{|x-y|^{n+2s}}\,dy \
{\mbox{ for }}x\in\R^n,$$
where $c_{n,s}$ is a suitable positive constant.
Moreover\footnote{In this paper, we focus
our attention on the case $n>4s$, since, under this assumption,
the $L^2$-theory developed in Section \ref{sec:est} becomes available.
It would be interesting to investigate the remaining cases.} 
we assume that $n>4s$, $\epsilon>0$ is a small parameter, $p=\frac{n+2s}{n-2s}$ is
the fractional critical Sobolev exponent, $0<q<p$ and $h$ is
a continuous function that
satisfies
\begin{eqnarray}
&& \omega:=supp\,h {\mbox{ is compact}}\label{h2}\\
{\mbox{and }} && h_+\not\equiv 0. \label{h3}
\end{eqnarray}
The structural assumption \eqref{h2} is quite important in our paper,
since it allows us to set up a convenient functional framework,
in which we consider perturbations of the standard solution for $\eps=0$
which remain positive inside the support of $h$ (if $h$ is not compactly
supported, this is not possible, since the standard solution approaches
zero at infinity).

Also, assumption \eqref{h3} says, roughly speaking,
that at least in some part of the space there is a
reaction term
to balance the (fractional) diffusion given by the principal part
of the equation.

More precisely, in this paper
we will find solutions of problem \eqref{problem} by considering it as a perturbation
of the equation
\begin{equation}\label{entire}
(-\Delta)^s u = u^{p} \ {\mbox{ in }}\R^n,
\end{equation}
with $p=\frac{n+2s}{n-2s}$.
It is known that the optimizers of the Sobolev embedding in $\mathbb{R}^n$ are unique,
up to translations and positive dilations. Namely if we set
\begin{equation}\label{zetazero}
z_0(x):=\alpha_{n,s}\frac{1}{(1+|x|^2)^{(n-2s)/2}},
\end{equation}
then all the optimizers of the Sobolev embedding
are obtained by the formula
\begin{equation}\label{sol}
z_{\mu,\xi}(x):= \mu^{(2s-n)/2}\,z_0\left(\frac{x-\xi}{\mu}\right),
\end{equation}
where $\mu>0$, $\xi\in\R^n$. The normalizing constant  $\alpha_{n,s}$  depends only on $n$ and $s$ (see \cite{Lieb}, \cite{Talenti},  \cite{DDS} and the references therein),
and the explicit value of~$\alpha_{n,s}$
is not particularly relevant in our framework.
Notice also that equation \eqref{entire} is the Euler Lagrange
equation of this Sobolev embedding.\medskip

It has been showed in \cite{DDS} that solutions to \eqref{entire} of the form \eqref{sol}
are nondegenerate. Namely, setting $\partial_\mu z_{\mu,\xi}$ and $\partial_\xi z_{\mu,\xi}$
the derivative of $z_{\mu,\xi}$ with respect to the parameters $\mu$ and $\xi$ respectively,
then all bounded solutions of the linear equation
$$ (-\Delta)^s \psi = p\,z_{\mu,\xi}^{p-1}\,\psi \ {\mbox{ in }}\R^n $$
are linear combinations of $\partial_\mu z_{\mu,\xi}$ and $\partial_\xi z_{\mu,\xi}$. We also refer to~\cite{frank}, where
the nondegeneracy result was proved
in detail for~$s=1/2$ and~$n=3$
(but the proof can be extended in higher dimensions
and for fractional exponents~$s\in(0,n/2)$ as well).

We set
$$ [u]_{\dot H^s(\R^n)}^2:= \iint_{\R^{2n}}\frac{|u(x)-u(y)|^2}{|x-y|^{n+2s}}\,dx\,dy, $$
and we define the space $\dot{H}^s(\R^n)$ as the completion of
the space of smooth and rapidly decreasing functions (the so-called
Schwartz space)
with respect to the norm $[u]_{\dot H^s(\R^n)}+\|u\|_{L^{\CRES}(\R^n)}$,
where
$$ \CRES=\frac{2n}{n-2s}$$
is the fractional critical exponent.

We observe that the homogeneous space $\dot H^s(\R^n)$
coincides with the space of functions $u\in L^{\CRES}(\R^n)$ with finite
seminorm $[\cdot]_{\dot H^s(\R^n)}$ (and the norm in $\dot H^s(\R^n)$
is also equivalent to the seminorm, due to Sobolev embedding).

We also introduce the space
$$ X^s:=\dot{H}^s(\R^n)\cap L^\infty(\R^n),$$
equipped with the norm
$$ \|u\|_{X^s}:=[u]_{\dot H^s(\R^n)} + \|u\|_{L^\infty(\R^n)}.$$

Given $f\in L^\beta(\R^n)$, where $\beta:=\frac{2n}{n+2s}$,
we say that $u\in X^s$ is a (weak) solution to
$(-\Delta)^s u=f$ in $\R^n$ if
$$ \iint_{\R^{2n}}
\frac{\big( u(x)-u(y)\big)\,\big( \varphi(x)-\varphi(y)\big)}{|x-y|^{n+2s}}
\,dx\,dy=\int_{\R^n} f\,\varphi\,dx,$$
for any $\varphi\in X^s$.

We prove the following:
\begin{theorem}\label{TH1}
Let
$s\in(0,1)$, $n>4s$ and $p=\frac{n+2s}{n-2s}$.
Suppose that $h$ is a continuous
function that satisfies \eqref{h2} and \eqref{h3}.
Assume also that
\begin{eqnarray}\label{ALT:1}
&&{\mbox{either $\frac{2s}{n-2s}<q<p$,}}
\\&&{\mbox{or $0<q\le \frac{2s}{n-2s}$ and $h\ge0$.}}\label{ALT:2}
\end{eqnarray}
Then, there exist $\epsilon_0>0$, $\mu_1>0$ and $\xi_1\in\R^n$ such that
problem \eqref{problem} has a positive solution $u_{1,\epsilon}$ for any $\epsilon\in(0,\epsilon_0)$,
and $u_{1,\epsilon}\rightarrow z_{\mu_1,\xi_1}$ in $X^s$ as $\epsilon\rightarrow 0$.\end{theorem}

\begin{theorem}\label{TH1BIS}
Let
$s\in(0,1)$, $n>4s$, $p=\frac{n+2s}{n-2s}$ and $\frac{2s}{n-2s}<q<p$.
Suppose that $h$ is a continuous
function that satisfies \eqref{h2} and \eqref{h3}, and that
changes sign.

Then for any $\epsilon\in(0,\epsilon_0)$ there exists
a second positive solution $u_{2,\epsilon}$ to \eqref{problem} that, as $\epsilon\rightarrow 0$,
converges in $X^s$ to $z_{\mu_2,\xi_2}$ with $\mu_2>0$, $\mu_2\neq\mu_1$,
and $\xi_2\in\R^n$, $\xi_2\neq\xi_1$.
\end{theorem}

In order to prove Theorems \ref{TH1} and \ref{TH1BIS}
we will use a Lyapunov-Schmidt reduction,
that takes advantage of the variational structure
of the problem.
Indeed, positive
solutions to \eqref{problem} can be found as critical points of the
functional $f_\epsilon: X^s\rightarrow\R$ defined by
\begin{eqnarray}\label{feps}
f_\epsilon(u) &:=&\frac{c_{n,s}}{4}\iint_{\R^{2n}}\frac{|u(x)-u(y)|^2}{|x-y|^{n+2s}}\,dx\,dy \\
&& \quad -\frac{\epsilon}{q+1}\,\int_{\R^n}h(x)\,u_+^{q+1}(x)\,dx
-\frac{1}{p+1}\,\int_{\R^n}u_+^{p+1}(x)\,dx.
\nonumber\end{eqnarray}
We notice that $f_\epsilon$ can be written as
\begin{equation}\label{perturbed}
f_\epsilon(u)= f_0(u)-\epsilon\, G(u),
\end{equation}
where
\begin{equation}\label{fzero}
f_0(u):= \frac{c_{n,s}}{4}\iint_{\R^{2n}}\frac{|u(x)-u(y)|^2}{|x-y|^{n+2s}}\,dx\,dy
-\frac{1}{p+1}\,\int_{\R^n}u_+^{p+1}(x)\,dx
\end{equation}
and
\begin{equation}\label{Gu}
G(u):= \frac{1}{q+1}\,\int_{\R^n}h(x)\,u_+^{q+1}(x)\,dx.
\end{equation}
Indeed, we will use a perturbation method that allows us to find critical points of $f_\epsilon$
by bifurcating from a manifold of critical points of the
unperturbed functional $f_0$
(see for instance \cite{AM} for the abstract method).

Notice that critical points of $f_0$ are solutions to \eqref{entire},
and so, in order to construct solutions to \eqref{problem},
we will start from functions of the form \eqref{sol}
and we will add a small error to them
in such a way that we obtain solutions to the perturbed problem.

This small error will be found by means of the Implicit Function Theorem.
To do this,
a crucial ingredient will be the nondegeneracy condition
proved in \cite{DDS} for $z_{\mu,\xi}$, but the application
of the linear theory in our case is non-standard
and it requires a pointwise control of the functional spaces.

Roughly speaking, one additional difficulty for us
is indeed that when~$q<1$ the energy functional is not
$C^2$ at the zero level set, and so
the classical Implicit Function Theorem cannot
be applied, unless we can avoid the singularity. For this,
the classical Hilbert space framework is not enough, and we have
to keep track of the pointwise behavior of the functions
inside our functional analysis framework. This is for instance the main reason
for which we work in
the more robust space~$X^s$ rather than in
the more classical space~$\dot H^s(\R^n)$.

Of course, the change of functional setting
produces some difficulties in the invertibility of
the operators, since the Hilbert-Fredholm theory does not directly apply,
and we will have to compensate it by an appropriate elliptic
regularity theory.

Once these difficulties are overcome,
the Lyapunov-Schmidt reduction allows us to reduce our problem
to the one of finding critical points of the perturbation $G$,
introduced in \eqref{Gu}. For this, we set
\begin{equation}\label{gamma intr}
\Gamma(\mu,\xi):= G(z_{\mu,\xi}),
\end{equation}
where $z_{\mu,\xi}$ has been introduced in \eqref{sol}.
The study of the behavior of $\Gamma$ will give us
the existence of critical points of $G$,
and so the existence of solution to \eqref{problem}.

We also mention \cite{two-bubble, erratum}, where the authors 
use a different reduction procedure to deal with a slightly supercritical problem 
in a bounded domain. 

\medskip

There is a huge literature concerning the search of solutions for
this kind of perturbative problems
in the classical case, i.e. when $s=1$ and the fractional Laplacian
boils down to the classical Laplacian, see
\cite{AGP, AGP2, ALM, AmbMal, AmbMal2, BM, CaW, Ci, Da, Mal, Mal2}.
In particular, Theorems \ref{TH1} and \ref{TH1BIS}
here can be seen as the
nonlocal counterpart\footnote{We take this opportunity to
point out that there is a flaw on the last formula of
page 28 of
\cite{AGP2}. Indeed, one cannot use Fatou Lemma there, since $h$
is not positive. The additional assumption
$\frac{2}{n-2}<q<p$ is needed in order to use the Dominated Convergence Theorem.}
of Theorem 1.3 in \cite{AGP2}.
See also \cite{GMP}, where the concave term appears for the first time.

In the fractional case, the situation is more involved.
Namely, 
the nonlocal Schr\"odinger equation
has recently received a growing attention not only
for the challenging mathematical difficulties that it offers,
but also due to some important physical applications
(see e.g.~\cite{LAS},
the appendix in~\cite{DDDV}, and the references therein).
In the subcritical case,
this nonlocal Schr\"odinger equation 
can be written as
$$ \epsilon^{2s}(-\Delta)^su+V(x)u=u^p \ {\mbox{ in }}\R^n,$$
with $1<p<\frac{n+2s}{n-2s}$ and~$V$ a smooth potential.
Multi-peak solutions for this type of equations
were considered recently in \cite{DDW}.
Also in this case, a key ingredient in the proof
is the uniqueness and nondegeneracy of the ground state
solution of the corresponding unperturbed problem,
which has been proved in~\cite{FLS}
for any~$s\in(0,1)$ and in any dimension,
after previous works in dimension~1 (see~\cite{FL}) and
for~$s$ close to~1 (see~\cite{FV}).

Moreover, given a bounded domain $\Omega\subset\R^n$, the Dirichlet problem
$$ \begin{cases}
\epsilon^{2s}(-\Delta)^su+u=u^p &\ {\mbox{ in }}\Omega,\\
u=0 &\ {\mbox{ in }}\R^n\setminus\Omega,
\end{cases} $$
was considered in \cite{DDDV},
where the authors constructed solutions that concentrate at the interior of the domain.\medskip

Concentrating solutions for fractional problems involving
critical or almost critical exponents were considered in \cite{CKL}.
See also \cite{CZ} for some concentration phenomena in particular cases
and~\cite{SEC} for the study of the soliton dynamics in related problems.
See also~\cite{CZ2} for a semilinear problem with critical power,
related to the scalar curvature problem,
that also exploits a Lyapunov-Schmidt reduction.
It is worth pointing out that, in our case, the presence of the subcritical,
possibly sublinear, power
in our problem introduces extra difficulties that have required
the development of certain elliptic regularity theory,
and the careful analysis of the corresponding functional framework.
Notice indeed that for sublinear powers~$q$ the energy functional
experiences a loss of regularity, so the standard functional analysis
methods are not directly available and several technical modifications
are needed.

In particular, we perform here a detailed analysis of the linearized equation,
that is the key ingredient to use the Lyapunov-Schmidt arguments.
We think that these results are of independent interest
and can be useful elsewhere.

\medskip

As a matter of fact, we point out that the nonlocal framework
considered here provides additional difficulties, in terms of the regularity theory
and for the perturbative arguments (for instance, the Lyapunov-Schmidt theory
and the invertibility of the linearized operators become more involved
in this setting, due to the nonlocal effects in the remainders).

We also notice that the positivity (or more generally,
the existence of a nontrivial positive component) of $h$,
as ensured by \eqref{h3} will allow us a qualitative analysis
on a reduced functional in Section \ref{sec:proof}.
\medskip

The paper is organized as follows. In Section \ref{sec:est}
we show some auxiliary fractional elliptic estimates needed in the
subsequent sections. In Section \ref{qless1} we perform the Lyapunov-Schmidt reduction,
with the detailed study of the linearized equation, and the associated functional analysis theory.
Section \ref{sec:gamma} is devoted to the study of the behavior of $\Gamma$,
as defined in \eqref{gamma intr}. Finally, in Section \ref{sec:proof} we
complete the proof of Theorems \ref{TH1} and \ref{TH1BIS}.

\section{Fractional elliptic estimates}\label{sec:est}

Here we obtain some uniform elliptic estimates
on Riesz potential (though the topic is of classical
flavor in harmonic analysis, we could not find in the literature
a statement convenient for our purposes). These estimates will be used in Section~\ref{qless1}
in order to obtain the continuity properties of our functionals.

We recall that
$$ H^s(\R^n)=\{u:\R^n\to\R {\mbox{ measurable }}
{\mbox{ s.t. }}\|u\|_{L^2(\R^n)}+[u]_{\dot H^s(\R^n)}<+\infty \}.$$

To start with, we recall the fractional Sobolev
inequality (see e.g. Theorem~6.5 in~\cite{DPV}):

\begin{lemma}\label{L SOB}
Let~$n>2s$.
Let~$f:\R^n\to\R$ be a measurable function. Suppose that there exists
a sequence of functions~$f_k\in H^s(\R^n)$
such that~$f_k\to f$ in~$\dot H^s(\R^n)$ and
a.e. in~$\R^n$. Then
\begin{equation}\label{SOB}
\|f\|_{L^{\CRES}(\R^n)} \le C\, [f]_{\dot H^s(\R^n)},
\end{equation}
for some~$C>0$ depending on~$n$ and~$s$.
In particular, the inequality in~\eqref{SOB}
holds true for any~$f\in X^s$.
\end{lemma}

Here is the fractional elliptic regularity needed for our goals:

\begin{theorem}\label{THABC}
Let~$n>4s$.
Let~$\beta:=2n/(n+2s)$ and~$\psi\in L^\beta(\R^n)$.
Let also
\begin{equation}\label{Jpsi} J\psi(x):=\int_{\R^n} \frac{\psi(y)}{|x-y|^{n-2s}}\,dy.\end{equation}
Then:
\begin{eqnarray}
\label{EL-1} && {\mbox{$J\psi\in L^{\CRES}(\R^n)$, and
$\| J\psi\|_{L^{\CRES}(\R^n)} \le C\,\|\psi\|_{L^\beta(\R^n)}$;}} \\
\label{EL-2} && {\mbox{$J\psi\in \dot{H}^s(\R^n)$, and
$[J\psi]_{\dot{H}^s(\R^n)} \le C\,\|\psi\|_{L^\beta(\R^n)}$;}} \\
\label{EL-4} && {\mbox{$(-\Delta)^s (J\psi) =c\psi$ in the weak
sense, i.e.}}\\ &&
\qquad\iint_{\R^{2n}}
\frac{\big( (J\psi)(x)-(J\psi)(y)\big)\,\big( \phi(x)-\phi(y)\big)}{|x-y|^{n+2s}}
\,dx\,dy
=c \int_{\R^n} \psi(x)\, \phi(x)\,dx \nonumber
\\ &&\qquad{\mbox{ for any $\phi\in X^s$;}}\nonumber\\
\label{EL-3} && {\mbox{if, in addition, it holds that~$\psi\in
L^{\frac{n}{2s}+\delta_o}(B_1)$,}}\\ \nonumber&& {\mbox{for some $\delta_o>0$, then
$J\psi\in L^{\infty}(\R^n)$,}}\\
\nonumber &&\qquad{\mbox{and
$\| J\psi\|_{L^\infty(\R^n)} \le C_{\delta_o}\,\Big(\|\psi\|_{
L^{\frac{n}{2s}+\delta_o}(B_1)}+
\|\psi\|_{L^\beta(\R^n)}\Big)$.}}\\
\nonumber &&{\mbox{In particular,
if~$\psi\in
L^\infty(\R^n)$, then
$J\psi\in L^{\infty}(\R^n)$,}}\\
\nonumber && \qquad{\mbox{and
$\| J\psi\|_{L^\infty(\R^n)} \le C\,\Big(\|\psi\|_{L^\infty(\R^n)}+
\|\psi\|_{L^\beta(\R^n)}\Big)$.}}
\end{eqnarray}
Here above, $C$ and $c$ are suitable positive
constants\footnote{In \label{P5} the sequel, for simplicity we will just take~$c=1$
in~\eqref{EL-4}. This can be accomplished simply by renaming~$J$ to~$c^{-1}J$.} only depending
on~$n$ and~$s$, while $C_{\delta_o}$ also depends on $\delta_o$.
\end{theorem}

We observe that $J$ above is the Riesz Potential.

\begin{proof}[Proof of Theorem \ref{THABC}] The claim in~\eqref{EL-1} follows from
an appropriate version of the
Hardy-Littlewood-Sobolev inequality, namely
Theorem~1
on page~119 of~\cite{stein:sing.int}, used here with~$\alpha:=2s$,
$p:=\beta$ and~$q:=\CRES$.

Now we take a sequence of smooth and rapidly decreasing functions~$\psi_j$
that converge to~$\psi$ in~$L^\beta(\R^n)$, and we set~$\Psi_j:=J\psi_j$.
We also set~$\Psi:=J\psi$. Thus, by~\eqref{EL-1}, we have that
$$ \| \Psi_j -\Psi\|_{L^{\CRES}(\R^n)}=
\| J(\psi_j -\psi)\|_{L^{\CRES}(\R^n)}\le C\|\psi_j-\psi\|_{L^{\beta}(\R^n)}\to0$$
as~$j\to+\infty$. Thus, up to a subsequence,
\begin{equation}\label{tyuifg890dfghj}
{\mbox{$\Psi_j\to\Psi$ a.e. in $\R^n$.}}
\end{equation}
Moreover, by a version of Parseval Identity (see e.g.
Lemma~2(b) in~\cite{stein:sing.int}),
we have that
\begin{equation}\label{new E psi-0}
\int_{\R^n} (J\psi_j)(x)\, \overline{g(x)}\,dx=c
\int_{\R^n} \hat\psi_{j}(\xi) \,|\xi|^{-2s}\,
\overline{\hat g(\xi)}\,d\xi,
\end{equation}
for some~$c>0$,
for every~$g$ that is smooth and rapidly decreasing (and possibly complex valued).
As standard, we have denoted by $\hat g=
{\mathcal{F}}g$ the Fourier transform of $g$.

Now, for any~$\phi$ smooth and rapidly decreasing
and any~$\delta>0$,
we take~$g_\delta$ to be the inverse Fourier transform of~$(|\xi|^2+\delta)^s
\hat \phi$, in symbols~$g_\delta:={\mathcal{F}}^{-1} \big( (|\xi|^2+\delta)^s
\hat \phi\big)$. We remark that $(|\xi|^2+\delta)^s\hat \phi$
is smooth and rapidly decreasing, hence so is~$g_\delta$.
Accordingly, \eqref{new E psi-0} implies that
\begin{equation}\label{very new E psi-0}
\int_{\R^n} (J\psi_j)(x)\, \overline{g_\delta (x)}\,dx=c
\int_{\R^n} \hat\psi_{j}(\xi) \,|\xi|^{-2s}\,\overline{(|\xi|^2+\delta)^s\hat \phi
(\xi)}\,d\xi.
\end{equation}
We claim that
\begin{equation}\label{678dd667uh}
{\mbox{$g_\delta \to
{\mathcal{F}}^{-1} (|\xi|^{2s}\hat \phi)$ in~$L^2(\R^n)$,
as $\delta\to0$.}}
\end{equation}
To check this, we use Plancherel Theorem to compute
\begin{equation}\label{7djdjjdfghqwe7aa}\begin{split}
& \| g_\delta- {\mathcal{F}}^{-1} (|\xi|^{2s}\hat \phi)\|^2_{L^2(\R^n)}
= \| \hat g_\delta- |\xi|^{2s}\hat \phi\|_{L^2(\R^n)}^2 \\
&\qquad=\big\|
[(|\xi|^2+\delta)^s-|\xi|^{2s}]\,\hat \phi\big\|_{L^2(\R^n)}^2
= \int_{\R^n} \big| (|\xi|^2+\delta)^s-|\xi|^{2s} \big|^2 \,|\hat \phi(\xi)|^2
\,d\xi.\end{split}
\end{equation}
Then we observe that, if~$\delta\in(0,1)$,
$$ \big| (|\xi|^2+\delta)^s-|\xi|^{2s} \big|^2 \le C\delta^{2s},$$
thus~\eqref{678dd667uh}
follows from~\eqref{7djdjjdfghqwe7aa}.

Moreover, since~$\psi_j$ is rapidly decreasing,
a direct computation with convolutions (see e.g. Lemma~5.1
in~\cite{DDDV}) gives that
\begin{equation} \label{yudvfi0bgdndndnd56}
|J\psi_j(x)|\le \frac{C_j}{1+|x|^{n-2s}},\end{equation}
for some~$C_j>0$. In particular, since~$n>4s$, we have that
\begin{equation}\label{J psi L2}
\Psi_j=J\psi_j\in L^2(\R^n).
\end{equation}
It is worth to point out that here is where the condition~$n>4s$
plays an important role.

As a matter of fact, the derivatives of~$\psi_j$ are rapidly decreasing
as well
and~$ \nabla\Psi_j =J(\nabla\psi_j)$, thus the
argument above also shows that~$\nabla\Psi_j\in L^2(\R^n,\R^n)$, and so
\begin{equation}\label{JH1}
\Psi_j \in H^1(\R^n).
\end{equation}
Using~\eqref{678dd667uh}, \eqref{J psi L2}
and the Plancherel Theorem, we conclude that
\begin{equation}\label{8dheief5tcvbvcx}
\begin{split}
&\lim_{\delta\to0}\int_{\R^n} (J\psi_j)(x)\, \overline{g_\delta (x)}\,dx
=\int_{\R^n} \Psi_j(x)\,
\overline{{\mathcal{F}}^{-1} (|\xi|^{2s}\hat \phi)(x)}\,dx\\ &\qquad=
\int_{\R^n} \hat\Psi_j(\xi)\,
\overline{|\xi|^{2s}\hat \phi(\xi)}\,d\xi=
\int_{\R^n} |\xi|^{2s} \hat\Psi_j(\xi)\,
\overline{\hat \phi(\xi)}\,d\xi.
\end{split}\end{equation}
Now we point out that, for~$\delta\in(0,1)$,
$$ \Big| |\xi|^{-2s}\,\overline{(|\xi|^2+\delta)^s \hat\phi(\xi)}\Big|\le
|\xi|^{-2s}\,(|\xi|^2+1)^s |\hat\phi(\xi)|$$
and this function is in~$L^1(\R^n)$, since~$n>2s$.
Accordingly, the Dominated Convergence Theorem gives that
$$\lim_{\delta\to0}
\int_{\R^n} \hat\psi_{j}(\xi) \,|\xi|^{-2s}\,\overline{(|\xi|^2+\delta)^s\hat \phi(\xi)}
\,d\xi=
\int_{\R^n} \hat\psi_{j}(\xi) \,\overline{\hat \phi(\xi)}
\,d\xi.$$
This, \eqref{very new E psi-0} and \eqref{8dheief5tcvbvcx}
imply that
\begin{equation}\label{very very new E psi-0}
\int_{\R^n} |\xi|^{2s} \hat\Psi_j (\xi)\,\overline{\hat\phi(\xi)}\,d\xi
=
c\int_{\R^n} \hat\psi_{j}(\xi) \,\overline{\hat \phi(\xi)}
\,d\xi=c\int_{\R^n} \psi_{j}(x) \,\phi(x)
\,dx
,\end{equation}
for any~$\phi$ smooth and rapidly decreasing.

Now we fix~$j\in\N$ and make use of \eqref{JH1}:
accordingly, by density,
we find a sequence~$\Psi_{j,k}$
of smooth and rapidly decreasing functions that converge to~$\Psi_j$
in $H^1(\R^n)$ as~$k\to+\infty$.

In particular, $\Psi_{j,k}\to \Psi_j$ in~$L^2(\R^n)$ and so,
by Plancherel Theorem, also~$\hat\Psi_{j,k}\to \hat\Psi_j$ in~$L^2(\R^n)$,
as $k\to+\infty$.
Moreover, $|\xi|^{2s}\le1$ if~$|\xi|\le1$ and~$|\xi|^{2s}\le|\xi|^2$
if~$|\xi|\ge1$, thus
\begin{equation}\label{7usdicvbbddsww} |\xi|^{2s}\le 1+|\xi|^2.\end{equation}
Consequently
\begin{equation}\begin{split}\label{aggiunto}
& \int_{\R^n} |\xi|^{2s} \big|\hat \Psi_{j,k}(\xi) -\hat \Psi_{j}(\xi)\big|^2\,d\xi\le
\int_{\R^n} (1+|\xi|^{2})\, \big|{\mathcal{F}} \big(\Psi_{j,k}(\xi) -\Psi_{j}(\xi)\big)\big|^2\,d\xi
\\ &\qquad\le C \| \Psi_{j,k} -\Psi_{j}\|_{H^1(\R^n)}^2 \to0
\end{split}\end{equation}
as~$k\to+\infty$, and therefore
$$ \lim_{k\to+\infty}\int_{\R^n} |\xi|^{2s}
\hat \Psi_j(\xi)\, \overline{\hat \Psi_{j,k}(\xi)}\,d\xi=
\int_{\R^n} |\xi|^{2s}
|\hat \Psi_j(\xi)|^2\,d\xi.$$
Then we
apply \eqref{very very new E psi-0}
with~$\phi:=\Psi_{j,k}$;
therefore
we see that
\begin{eqnarray*}
&& \int_{\R^n} |\xi|^{2s}
|\hat\Psi_j(\xi)|^2\,d\xi=\lim_{k\to+\infty}\int_{\R^n} |\xi|^{2s}
\hat\Psi_j(\xi)\, \overline{\hat\Psi_{j,k}(\xi)}\,d\xi
\\&&\qquad =\lim_{k\to+\infty}
c\int_{\R^n} \hat\psi_{j}(\xi) \,\overline{\hat \Psi_{j,k}(\xi)}\,d\xi
=
c\int_{\R^n} \hat\psi_{j}(\xi) \,\overline{\hat \Psi_{j}(\xi)}\,d\xi
.\end{eqnarray*}
Thus, by
the H\"older Inequality
with exponents~$\beta$ and~$2n/(n-2s)$, we obtain
\begin{eqnarray*}
&& \int_{\R^n} |\xi|^{2s}
|\hat\Psi_j(\xi)|^2\,d\xi =
c\int_{\R^n} \psi_{j}(\xi) \, \Psi_{j}(\xi)\,d\xi\\&&\qquad
\le c\,\|\psi_j\|_{L^\beta(\R^n)} \,\|\Psi_j\|_{L^{\CRES}(\R^n)}
\le C\,\|\psi_j\|_{L^\beta(\R^n)}^2,\end{eqnarray*}
where~\eqref{EL-1} was used in the last step.

This (together with the equivalence of the seminorm in $H^s(\R^n)$,
see Proposition 3.4 in \cite{DPV}) says that
$$ \iint_{\R^{2n}} \frac{|\Psi_j(x)-\Psi_j(y)|^2
}{|x-y|^{n+2s}}\,dx\,dy\le C \|\psi_j\|_{L^\beta(\R^n)}^2.$$
So we recall~\eqref{tyuifg890dfghj}
and we take limit as~$j\to+\infty$,
obtaining, by Fatou Lemma and the fact that $\psi_j\to\psi$ in $L^\beta(\R^n)$,
that
$$ \iint_{\R^{2n}} \frac{|\Psi(x)-\Psi(y)|^2
}{|x-y|^{n+2s}}\,dx\,dy\le C \|\psi\|_{L^\beta(\R^n)}^2,$$
that establishes~\eqref{EL-2}.

Now we prove~\eqref{EL-4}.
For this, we use~\eqref{EL-2} to see that
\begin{eqnarray*}
&& \iint_{\R^{2n}}
\frac{\big| (\Psi_j-\Psi)(x)-(\Psi_j-\Psi)(y)\big|^2}{|x-y|^{n+2s}}
\,dx\,dy
=[\Psi_j-\Psi]^2_{\dot H^s(\R^n)}\\ &&\qquad=
[J(\psi_j-\psi)]^2_{\dot H^s(\R^n)}\le C\,\|\psi-\psi_j\|^2_{L^\beta(\R^n)}\to0\end{eqnarray*}
as~$j\to+\infty$.
This says that the sequence of functions
$$ M_j(x,y):=
\frac{\Psi_j(x)-\Psi_j(y)}{|x-y|^{\frac{n+2s}2}}$$
converges to the function
$$ M(x,y):=
\frac{\Psi(x)-\Psi(y)}{|x-y|^{\frac{n+2s}2}}$$
in $L^2(\R^{2n})$.
In particular, this implies weak convergence in~$L^2(\R^{2n})$,
that is
$$ \lim_{j\to+\infty}\iint_{\R^{2n}} M_j(x,y)
\,{\gamma(x,y)}\,dx\,dy= \iint_{\R^{2n}} M(x,y)
\,{\gamma(x,y)}\,dx\,dy$$
for any~$\gamma\in L^2(\R^{2n})$.

Thus, if~$\phi$ is smooth and rapidly decreasing,
we can take
$$ \gamma(x,y):=
\frac{\phi(x)-\phi(y)}{|x-y|^{\frac{n+2s}2}}$$
and obtain that
\begin{eqnarray*}&& \lim_{j\to+\infty} \iint_{\R^{2n}}
\frac{\big(\Psi_j(x)-\Psi_j(y)\big)\,\big(\phi(x)-\phi(y)
\big)}{|x-y|^{n+2s}}\,dx\,dy \\ &&\qquad=
\iint_{\R^{2n}}
\frac{\big(\Psi(x)-\Psi(y)\big)\,\big(\phi(x)-\phi(y)
\big)}{|x-y|^{n+2s}}\,dx\,dy.\end{eqnarray*}
Moreover, since $\psi_j$ converges to $\psi$ in $L^\beta(\R^n)$, we have that
$$ \lim_{j\to+\infty}\int_{\R^n}\psi_j(x)\,\phi(x)\,dx=\int_{\R^n}\psi(x)\,\phi(x)\,dx.$$
Consequently,
we can pass to the limit~\eqref{very very new E psi-0}
and obtain~\eqref{EL-4} for any~$\phi$ which is smooth and rapidly
decreasing.

It remains to establish~\eqref{EL-4} for any~$\phi\in X^s$.
For this, we fix~$\phi\in X^s$
and we take a sequence~$\phi_k$ of
smooth and rapidly
decreasing functions that converge to~$\phi$ in~$\dot{H}^s(\R^n)$,
and so, by Lemma \ref{L SOB},
also in~$L^{\CRES}(\R^n)$.
Also, we know that~$\Psi\in \dot H^s(\R^n)$, thanks to~\eqref{EL-2}.
In particular, by Cauchy-Schwarz and H\"older inequalities,
we obtain that
\begin{eqnarray*}
&& \left|
\iint_{\R^{2n}}
\frac{\big(\Psi(x)-\Psi(y)\big)\,\big((\phi-\phi_k)(x)-(\phi-\phi_k)(y)
\big)}{|x-y|^{n+2s}}\,dx\,dy\right|
\\&&\qquad\le [\Psi]_{\dot H^s(\R^n)}\, [\phi-\phi_k]_{\dot H^s(\R^n)}\to0\\
&&{\mbox{ and }}
\left| \int_{\R^n} \psi(x)\, \Big(\phi(x)-\phi_k(x)\Big)\,dx\right|
\le \|\psi\|_{L^\beta(\R^n)}\,\|\phi-\phi_k\|_{L^{\CRES}(\R^n)}\to0
\end{eqnarray*}
as~$k\to+\infty$. Therefore,
we can write~\eqref{EL-4} for the smooth and rapidly
decreasing functions~$\phi_k$, pass to the limit in~$k$,
and so obtain~\eqref{EL-4} for~$\phi\in X^s$.
This completes the proof of~\eqref{EL-4}.

Now we prove~\eqref{EL-3}. For this, we use the H\"older Inequality
with exponents~$\frac{n}{2s}+\delta_o$ and $\frac{n+2s\delta_o}{n-2s(1-\delta_o)}$
and with exponents~$\beta$ and~$2n/(n-2s)$
to calculate
\begin{eqnarray*}
|J\psi(x)| &\le&
\int_{\R^n} \frac{|\psi(x-y)|}{|y|^{n-2s}}\,dy\\
&\le& \left[\int_{B_1}|\psi(x-y)|^{ \frac{n}{2s}+\delta_o}\,dy\right]^{
\frac{2s}{n+2s\delta_o}}\,\left[
\int_{B_1} \frac{dy}{|y|^{\frac{(n-2s)(n+2s\delta_o)}{n-2s(1-\delta_o)}}}\,dy\right]^{ \frac{n-2s(1-\delta_o)}{n+2s\delta_o}}
\\ &&+\left[ \int_{\R^n\setminus B_1} |\psi(x-y)|^\beta\,dy\right]^{\frac{1}{\beta}}
\,\left[
\int_{\R^n\setminus B_1} \frac{dy}{|y|^{2n}}\,dy\right]^{ \frac{n-2s}{2n} }
\\ &\le& C_{\delta_o}\,\Big( \|\psi\|_{L^{\frac{n}{2s}+\delta_o}(B_1)}+
\|\psi\|_{L^\beta(\R^n)}\Big),
\end{eqnarray*}
and this establishes~\eqref{EL-3}.
\end{proof}

We establish now a generalization of Theorem 8.2 in \cite{DFV},
that will provide us an $L^\infty$ estimate for the solutions
of some general kind of subcritical and critical problems in $\mathbb{R}^n$.

\begin{theorem}\label{LinfResult}
Let~$f$, $f_1,\cdots, f_K:\R^n\times\R\to\R$ be such that
\begin{equation*}
\begin{split}
& |f(x,r)|\le \sum_{i=1}^K f_i(x,r)\\
{\mbox{with }}\; &f_i(x,r)\le h_i(x) \,|r|^{\gamma_i}
\end{split}\end{equation*}
where
\begin{equation}\label{DIST}\begin{split}
& \gamma_1,\cdots,\gamma_K \in [0,\CRES-1)\\
{\mbox{and }}\;
& h_1,\cdots,h_K \in L^{m_i}(\R^n,[0,+\infty)), \; {\mbox{ with }} \;
m_i \in \left( \underline{m}_i,\,+\infty\right]\\
&\qquad{\mbox{ where }}
\underline{m}_i:=\left\{ \begin{matrix}
\displaystyle\frac{\CRES}{\CRES-2} & {\mbox{ if }} \gamma_i\in[0,\,1]\\
&\\
\displaystyle\frac{\CRES}{\CRES-1-\gamma_i}& 
{\mbox{ if }} \gamma_i\in(1,\,\CRES-1).
\end{matrix}\right.
\end{split}\end{equation}
Let~$u\in \dot{H}^s(\mathbb{R}^n)$ be a weak solution of
\begin{equation*}
(-\Delta)^s u= f(x,u(x)) \quad\hbox{in }\mathbb{R}^n.
\end{equation*}
Then
$$\|u\|_{L^\infty(\Omega)}\le C,$$ 
where~$C>0$ depends on~$n$, $s$, $\|u\|_{L^{\CRES}(\R^n)}$, $\gamma_i$, $m_i$
and~$\|h_i\|_{L^{m_i}(\R^n)}$.
\end{theorem}

\begin{proof} We will prove a stronger statement, namely that
if~$(-\Delta)^s u\le f(x,u(x))$ in the weak sense
and~$f(x,r)\le \sum_{i=1}^K f_i(x,r)$, with~$f_i$ as above, then~$u$
is bounded from above (the bound from below 
when~$(-\Delta)^s u\ge f(x,u(x))$ can be obtained similarly
under the corresponding growth assumptions).

To prove the desired bound on~$u$
we will use an argument that goes back to Stampacchia. 
Throughout the proof 
\begin{equation}\label{CCCACC}\begin{split}&
{\mbox{we will denote by~$C>0$ a quantity that}}\\&{\mbox{may
depend on~$n$, $s$, $m_i$, $\gamma_i$, $\|u\|_{L^{\CRES}(\R^n)}$
and $\|h_i\|_{L^{m_i}(\R^n)}$,}}\end{split}
\end{equation}
and which may change from line to line.

Notice that if~$u$ vanishes identically then the claim trivially follows, 
therefore we assume that~$u$ does not vanish identically.

Also, we rewrite the condition on~$m_i$
in~\eqref{DIST} as
\begin{equation}\label{DD} 0\le \frac{1}{m_i}<
\left\{ \begin{matrix}
\displaystyle\frac{\CRES-2}{\CRES} & {\mbox{ if }} \gamma_i\in[0,\,1]\\
&\\
\displaystyle\frac{\CRES-1-\gamma_i}{\CRES}& 
{\mbox{ if }} \gamma_i\in(1,\,\CRES-1).
\end{matrix}\right.\end{equation}
In any case
\begin{equation}\label{DIST2} \frac{1}{m_i}<
\frac{\CRES-1-\gamma_i}{\CRES}.\end{equation}
Now, we set
$$\Theta_i := \frac{1}{m_i}+\frac{2-\CRES+\gamma_i}{\CRES}\in\R,$$
and we claim that
\begin{equation}\label{CHE}
\Theta_i <\frac{\gamma_i}{\CRES}\,\min\{1,\,\gamma_i^{-1}\}.
\end{equation}
Indeed, if~$\gamma_i\in[0,\,1]$,by~\eqref{DD}
we have that
$$ \Theta_i < \frac{\CRES-2}{\CRES} +\frac{2-\CRES+\gamma_i}{\CRES}
=\frac{\gamma_i}{\CRES}
=\frac{\gamma_i}{\CRES}\,\min\{1,\,\gamma_i^{-1}\}$$
which implies~\eqref{CHE} in this case. 
If~$\gamma_i\in(1,\,\CRES-1)$, again by~\eqref{DD} we get
$$ \Theta_i < 
\frac{\CRES-1-\gamma_i}{\CRES}+\frac{2-\CRES+\gamma_i}{\CRES}
= \frac{1}{\CRES}=
\frac{\gamma_i\,\gamma_i^{-1}}{\CRES}=
\frac{\gamma_i}{\CRES}\,\min\{1,\,\gamma_i^{-1}\},$$
which completes the proof of~\eqref{CHE}.

We observe that \eqref{DD} also implies that 
when~$\gamma_i=0$ then~$\Theta_i<0$, and so~$\Theta_i/\gamma_i=-\infty$. Thus~\eqref{CHE} gives that
$$ \frac{\CRES}{\gamma_i}\Theta_i
<  \min\{ 1, \,\gamma_i^{-1}\}.$$
Hence, we can introduce an additional set of parameters~$a_i$, 
fixed arbitrarily such that
\begin{equation}\label{89.1}
\max\left\{0,\,\frac{\CRES}{\gamma_i}\Theta_i\right\}
< a_i\le \min\{ 1, \,\gamma_i^{-1}\}.
\end{equation}
We notice that
\begin{equation}\label{89.2}
a_i\in [0,\,1],\qquad \gamma_i a_i\le 1
\end{equation}
and
\begin{equation}\label{89.3}
\frac{\gamma_i a_i}{\CRES} > \Theta_i
= \frac{1}{m_i}+\frac{2-\CRES+\gamma_i}{\CRES}.
\end{equation}

Now, let $0<\delta<1$ to be chosen later, and define
\begin{equation}\label{def aggiunta}
\phi(x):=\frac{\delta u(x)}{\|u\|_{L^{\CRES}(\mathbb{R}^n)}},\quad x\in\mathbb{R}^n.
\end{equation}
Thus, using the notation in \eqref{CCCACC}, we can write
\begin{equation}\label{notation}
\phi=C\delta u.
\end{equation}
Also, 
\begin{equation}\label{0990990}
\|\phi\|_{L^{\CRES}(\mathbb{R}^n)}=\delta.
\end{equation}
Moreover, $\phi$ solves weakly
\begin{equation}\label{eqPhi}
(-\Delta)^s\phi \le g \quad \hbox{in }\mathbb{R}^n,
\end{equation}
where
$$ g(x):=\frac{\delta}{\|u\|_{L^{\CRES}(\R^n)}}f(x,u(x)).$$
We observe that
\begin{equation}\label{PG1}\begin{split}
& g(x)\le \sum_{i=1}^K g_i(x),\\
{\mbox{with }}&g_i(x):=
\frac{\delta}{\|u\|_{L^{\CRES}(\R^n)}} \,f_i(x,u(x))
\le C\delta h_i(x) \,|u(x)|^{\gamma_i}.
\end{split}\end{equation}

Now, for every integer $k\in\mathbb{N}$,
let us define $A_k:=1-2^{-k}$ and the functions
$$w_k(x):=(\phi(x)-A_k)^+,\quad\hbox{ for every }x\in\mathbb{R}^n.$$
By construction, $w_k\in\dot{H}^s(\mathbb{R}^n)$ and 
\begin{equation}\label{ww}
w_{k+1}(x)\leq w_k(x) \quad {\mbox{ a.e. in }}\mathbb{R}^n.
\end{equation}
Moreover, following \cite{DFV} it can be checked that for any $k\in\mathbb{N}$,
\begin{equation}\label{subsets}
\{w_{k+1}>0\}\subseteq \{w_k>2^{-(k+1)}\}
\end{equation}
and
\begin{equation}\label{boundPhi}
0<\phi(x)<2^{k+1}w_k(x)\quad\hbox{ for any }x\in \{w_{k+1}>0\}.
\end{equation}
Consider now
\begin{equation}\label{Ukkk}
U_k:=\|w_k\|_{L^{\CRES}(\mathbb{R}^n)}^{\CRES}.
\end{equation}
Thus, applying (8.10) of \cite{DFV} with $v:=\phi-A_{k+1}$ we obtain
\begin{equation}\begin{split}\label{wqpgftuiqweg}
[w_{k+1}]_{\dot{H}^s(\mathbb{R}^n)}^2=&\iint_{\mathbb{R}^{2n}}{\frac{|w_{k+1}(x)-w_{k+1}(y)|^2}{|x-y|^{n+2s}}\,dx\,dy}\\
\leq&\iint_{\mathbb{R}^{2n}}{\frac{(\phi(x)-\phi(y))(w_{k+1}(x)-w_{k+1}(y))}{|x-y|^{n+2s}}\,dx\,dy}.
\end{split}\end{equation}

Now, from~\eqref{boundPhi} we have that~$|\phi|<2^{k+1}w_k$ in~$\{w_{k+1}>0\}$, 
and so~$\delta|u|<C^kw_k$, thanks to~\eqref{notation}. 
Therefore, using the parameters~$a_i$ introduced in~\eqref{89.1}
and~\eqref{89.2}, we can estimate~$g_i$ given in~\eqref{PG1}
as
\begin{eqnarray*}
g_i &\le& C\delta^{1-a_i\gamma_i}\; h_i
\;(\delta\,|u|)^{a_i\gamma_i} \;|u|^{(1-a_i)\gamma_i} \\
&\le& C^k \delta^{1-a_i\gamma_i} \; h_i\;
w_k^{a_i\gamma_i}\;|u|^{(1-a_i)\gamma_i}\\
&\le& C^k \, h_i\;
w_k^{a_i\gamma_i}\;|u|^{(1-a_i)\gamma_i}
\end{eqnarray*}
in the set~$\{w_{k+1}>0\}$.
Hence, we use $w_{k+1}$ as a test function in \eqref{eqPhi} and we obtain that
\begin{equation}\begin{split}\label{estwk}
&\iint_{\mathbb{R}^{2n}}{\frac{(\phi(x)-\phi(y))(w_{k+1}(x)-w_{k+1}(y))}{|x-y|^{n+2s}}\,dx\,dy} \leq \int_{\R^n}g(x)\,w_{k+1}(x)\,dx \\
&\quad \leq C^k \sum_{i=1}^{K} \int_{  \{w_{k+1}>0\}  } 
h_i(x)\,w_k^{a_i\gamma_i}(x)\,w_{k+1}(x)\,|u(x)|^{(1-a_i)\gamma_i}\,dx. 
\end{split}\end{equation}
This and~\eqref{wqpgftuiqweg} imply that 
$$ [w_{k+1}]_{\dot{H}^s(\mathbb{R}^n)}^2\le C^k \sum_{i=1}^{K} \int_{  \{w_{k+1}>0\}  } h_i(x)\,w_k^{a_i\gamma_i}(x)\,w_{k+1}(x)\,|u(x)|^{(1-a_i)
{\gamma_i}}\,dx,$$
and so, recalling~\eqref{Ukkk}, we get 
\begin{equation}\label{qpquiwnd}
c\,U_{k+1}^{2/\CRES}\le  C^k \sum_{i=1}^{K} \int_{  \{w_{k+1}>0\}  } 
h_i(x)\,w_k^{a_i\gamma_i}(x)\,w_{k+1}(x)\,|u(x)|^{(1-a_i)\gamma_i}\,dx.
\end{equation}
In order to estimate the right hand side of~\eqref{qpquiwnd}, 
we introduce a new set of parameters: 
we recall~\eqref{DIST2} and obtain that
\begin{equation}\label{XII}
\xi_i := \frac{\CRES-1-\gamma_i}{\CRES}-\frac{1}{m_i} =1-\frac{1+\gamma_i}{\CRES}-
\frac{1}{m_i}\in (0,1).\end{equation}
Therefore, using~\eqref{ww}
and the H\"older inequality with exponents
$$ m_i, \quad \frac{\CRES}{1+a_i\gamma_i},\quad
\frac{\CRES}{(1-a_i)\gamma_i},\quad \frac{1}{\xi_i}$$
we obtain that, for any~$i\in\{1,\cdots,K\}$,
\begin{equation}\label{P0}
\begin{split}
& \int_{\{w_{k+1}>0\}}
h_i(x)\,
w_k^{a_i\gamma_i}(x)\,w_{k+1}(x)\,|u(x)|^{(1-a_i)\gamma_i} \,dx\\
& \quad\le\int_{\{w_{k+1}>0\}}
h_i(x)\,
w_k^{1+a_i\gamma_i}(x)\,|u(x)|^{(1-a_i)\gamma_i} \,dx\\
&\quad\le 
\left[\int_{\{w_{k+1}>0\}} h_i^{m_i}(x)\,dx\right]^{\frac{1}{m_i}}
\,\left[\int_{\{w_{k+1}>0\}}w_k^{\CRES}(x)
\,dx\right]^{\frac{1+a_i\gamma_i}{\CRES}} \,
\\ &\quad\cdot\left[\int_{\{w_{k+1}>0\}}
|u(x)|^{\CRES} \,dx\right]^{\frac{(1-a_i)\gamma_i}{\CRES}}
\left[\int_{\{w_{k+1}>0\}} 1\,dx\right]^{\xi_i}\\
&\quad\leq \|h_i\|_{L^{m_i}(\R^n)}\; U_k^{\frac{1+a_i\gamma_i}{\CRES}}
\;\|u\|_{L^{\CRES}(\R^n)}^{(1-a_i)\gamma_i}\;
\big| \{w_{k+1}>0\} \big|^{\xi_i}.
\end{split}
\end{equation}
On the other hand, by \eqref{subsets}
\begin{eqnarray*}
U_k&=&\|w_k\|_{L^{\CRES}(\mathbb{R}^n)}^{\CRES}\geq \int_{\{w_k>2^{-(k+1)}\}}{w_k^{\CRES}\,dx}\\
&\geq& 2^{-\CRES(k+1)}|\{w_{k}>2^{-(k+1)}\}|\geq 2^{-\CRES(k+1)}|\{w_{k+1}>0\}|,
\end{eqnarray*}
and thus,
$$
|\{w_{k+1}>0\}|^{\xi_i}\leq (2^{\CRES(k+1)}U_k)^{\xi_i}\le C^k U_k^{\xi_i}.
$$
Using this in~\eqref{P0} we have
\begin{equation}\label{TO2Y}
\int_{\{w_{k+1}>0\}}
h_i(x)\,
w_k^{a_i\gamma_i}(x)\,w_{k+1}(x)\,|u(x)|^{(1-a_i)\gamma_i} \,dx
\le C^k \, U_k^{\tau_i},
\end{equation}
with
$$ \tau_i:=\frac{1+a_i\gamma_i}{\CRES}+\xi_i.$$
Notice that~\eqref{XII} and~\eqref{89.3} imply that
\begin{equation}\label{TO9}\begin{split}
\tau_i \,&=\frac{1+a_i\gamma_i}{\CRES}+
\frac{\CRES-1-\gamma_i}{\CRES}-\frac{1}{m_i}\\
&= \frac{a_i\gamma_i}{\CRES}+\frac{\CRES-\gamma_i}{\CRES}-\frac{1}{m_i}\\
&> \frac{1}{m_i}+\frac{2-\CRES+\gamma_i}{\CRES}
+\frac{\CRES-\gamma_i}{\CRES}-\frac{1}{m_i}\\
&=\frac{2}{\CRES}.\end{split}
\end{equation}
Thus, inserting~\eqref{TO2Y} into~\eqref{qpquiwnd} we obtain that
\begin{equation}\label{TO5}
U_{k+1}^{2/\CRES} \le C^k \sum_{i=1}^{K} U_k^{\tau_i},\end{equation}
up to renaming~$C$. We observe that
\begin{equation}\label{step 0}
U_0=\|\phi^+\|_{L^{\CRES}(\R^n)}^{\CRES}\le \|\phi\|_{L^{\CRES}(\R^n)}^{\CRES}=\delta^{\CRES},
\end{equation}
thanks to~\eqref{0990990}. 
As a matter of fact, we have that~$U_k\le U_0\le 1$, so we can define
\begin{equation}\label{tau-bis}\tau:=
\min\{\tau_1,\dots,\tau_N\}.\end{equation} 
Hence~\eqref{TO5} becomes
\begin{equation*}
U_{k+1}^{2/\CRES} \le C^k U_k^{\tau},\end{equation*}
and so
\begin{equation}\label{TO55}
U_{k+1}\le C^k U_k^{\vartheta},\end{equation}
with~$\vartheta:= \CRES\tau/2$, after renaming~$C$.
Notice that by~\eqref{TO9}
and~\eqref{tau-bis}
$$ \vartheta>1.$$  
This, \eqref{TO55} and~\eqref{step 0} imply that
$$ \lim_{k\to+\infty} U_k=0$$
as long as~$\delta$ is fixed sufficiently small (in dependence of
the above~$C$). 
Moreover, since $0\leq w_k\leq |\phi|\in L^{\CRES}(\mathbb{R}^n)$ for any $k\in\mathbb{N}$
and $\displaystyle\lim_{k\rightarrow\infty}{w_k}=(\phi-1)^+$ a.e. in $\mathbb{R}^n$, by the Dominated Convergence Theorem we get
$$\lim_{k\rightarrow\infty}{U_k}=\|(\phi-1)^+\|_{L^{\CRES}(\mathbb{R}^n)}^{\CRES}=0,$$
and therefore $\phi\leq 1$ a.e. in $\mathbb{R}^n$.
Thus, recalling the definition of $\phi$ in~\eqref{def aggiunta}, we conclude that
$$u(x)\leq \frac{\|u\|_{L^{\CRES}(\mathbb{R}^n)}}{\delta},$$
with $\delta\in (0,1)$ fixed.
This concludes the proof of Theorem~\ref{LinfResult}.
\end{proof}

\section{The Lyapunov-Schmidt reduction}\label{qless1}
In this section we perform the Lyapunov-Schmidt reduction.
Since the argument is delicate and involves many lemmata,
we prefer to develop it in different steps.

\subsection{Preliminaries on the functional setting}
Given $0<\mu_1<\mu_2$ and $R>0$,
we define the manifold
\begin{equation}\label{critman}
Z_0:=\{z_{\mu,\xi} {\mbox{ s.t. }} \mu_1 <\mu <\mu_2, \ |\xi|< R\},
\end{equation}
where $z_{\mu,\xi}$ was introduced in \eqref{sol}.
We will perform our choice of $R$, $\mu_1$ and $\mu_2$ later on.
Notice that the functions in $Z_0$ are critical points of $f_0$,
as defined in \eqref{fzero}.

We will often implicitly identify $Z_0$ with the
subdomain $(\mu_1,\mu_2)\times B_R$
of $\R^{n+1}$
described by coordinates $(\mu,\xi)$.

In order to apply the abstract variational method
discussed in the introduction, we would
need in principle
the functional $f_\epsilon$ defined in \eqref{feps}
to be $C^2$ on $\dot{H}^s(\R^n)$.
Unfortunately, this is not true if $q<1$, and therefore,
in order to treat the whole set of values $q\in(0,p)$,
we recall that~$\omega$ is the support of the function~$h$ and we set
\begin{eqnarray}
a &:=& \inf \{ z_{\mu,\xi}(x) {\mbox{ s.t. }} x\in\omega, \ \mu_1<\mu<\mu_2, \ |\xi|<R\},
\nonumber\\
V &:=& \{w\in X^s {\mbox{ s.t }} \|w\|_{X^s}<a/2\} \nonumber\\
{\mbox{and }} U &:= &\left\{u:=z_{\mu,\xi}+w {\mbox{ s.t. }} z_{\mu,\xi}\in Z_0, \ w\in V\right\}.
\label{defU}
\end{eqnarray}
We observe that, if $u\in U$ and $x\in\omega$, then
$$ u(x)=z_{\mu,\xi}(x)+w(x)\ge a-\|w\|_{L^\infty(\R^n)}\ge a-\|w\|_{X^s}
>a-\frac{a}{2}=\frac{a}{2},$$
and so
\begin{equation}\label{PP8} u(x)>\frac{a}{2}>0 \ {\mbox{ for any }} x\in\omega. \end{equation}
Therefore, recalling \eqref{Gu}, we obtain that the functional $G$ is $C^2$ on $U$.
Hence, also $f_\epsilon:U\rightarrow\R$ is of class $C^2$.

Now, we set
\begin{equation}\label{qj}
 q_j:=\frac{\partial z_{\mu,\xi}}{\partial\xi_j}, \ j=1,\ldots,n, \ {\mbox{ and }} \
q_{n+1}:=\frac{\partial z_{\mu,\xi}}{\partial\mu},
\end{equation}
and we notice that~$q_j$ satisfies
\begin{equation}\label{EQqj}
(-\Delta)^s q_j=pz_{\mu,\xi}^{p-1}q_j  \ {\mbox{ in }} \R^n
\end{equation}
for every~$j=1,\dots,n+1$. We also denote by
$$ T_{z_{\mu,\xi}}Z_0 := span\left\lbrace q_1,\ldots,q_{n+1} \right\rbrace $$
the tangent space to $Z_0$ at $z_{\mu,\xi}$.

Moreover, $\langle \cdot, \cdot\rangle$ denotes the scalar product in $\dot{H}^s(\R^n)$,
that is, for any $v_1,v_2\in\dot{H}^s(\R^n)$,
$$ \langle v_1, v_2\rangle=
\iint_{\R^{2n}} \frac{\big( v_1(x)-v_1(y)\big)\,\big(
v_2(x)-v_2(y)\big)}{|x-y|^{n+2s}}\,dx\,dy.$$
We also define the notion of orthogonality with respect to such scalar product
and we denote it by $\perp$. That is, we set
$$ \left(T_{z_{\mu,\xi}}Z_0\right)^{\perp} :=
\Big\{ v \in \dot H^s(\R^n) {\mbox{ s.t. }}
\langle v, \phi\rangle =0 {\mbox{ for all }}
\phi \in T_{z_{\mu,\xi}}Z_0
\Big\}.$$
In particular, we prove the following orthogonality result.

\begin{lemma}\label{eigen}
There exist~$\lambda_i=\lambda_i(\mu,\xi)$, for~$i=1,\dots,n+1$,
such that
$$ \langle q_i,q_j\rangle =\left\{
\begin{matrix}
0 & {\mbox{ if }} \ i\ne j,\\
\lambda_i& {\mbox{ if }} \ i=j,
\end{matrix}
\right. $$
and
$$ \inf_{ {{\mu\in(\mu_1,\mu_2)}\atop{|\xi|<R}}\atop{i\in\{1,\dots,n+1\}} }
\lambda_i(\mu,\xi)>0.$$
\end{lemma}

\begin{proof} For any~$r\ge0$, we write
$$ \bar z(r):=\frac{\alpha_{n,s}}{(1+r)^{(n-2s)/2}}.$$
In this way, recalling the definition in \eqref{zetazero},
we have that~$z_0(x)=\bar z(|x|^2)$ and so
$$ z_{\mu,\xi}(x)=\mu^{(2s-n)/2} \bar z \left(\frac{|x-\xi|^2}{\mu^2}\right).$$
So we obtain that
$$ \frac{\partial z_{\mu,\xi}}{\partial\xi_i}(x)=
\mu^{(2s-n)/2} \bar z' \left( \frac{|x-\xi|^2}{\mu^2}\right)
\frac{2(\xi_i-x_i)}{\mu^2}$$
and therefore
$$ \frac{\partial z_{\mu,\xi}}{\partial\xi_i}(y+\xi)=
\mu^{(2s-n)/2} \bar z' \left( \frac{|y|^2}{\mu^2}\right)
\frac{2(-y_i)}{\mu^2},$$
which is odd in the variable~$y_i$.

Similarly,
$$ \frac{\partial z_{\mu,\xi}}{\partial\mu}(x)=
\frac{2s-n}{2} \mu^{(2s-n-2)/2} \bar z \left(\frac{|x-\xi|^2}{\mu^2}\right)
-\mu^{(2s-n)/2} \bar z'\left(\frac{|x-\xi|^2}{\mu^2}\right)
\frac{2|x-\xi|^2}{\mu^3},$$
thus
\begin{equation}\label{RHS99} \frac{\partial z_{\mu,\xi}}{\partial\mu}(y+\xi)=
\frac{2s-n}{2} \mu^{(2s-n-2)/2} \bar z \left(\frac{|y|^2}{\mu^2}\right)
-\mu^{(2s-n)/2} \bar z'\left(\frac{|y|^2}{\mu^2}\right)
\frac{2|y|^2}{\mu^3},\end{equation}
that is even in any of the variables~$y_i$.

\noindent Notice also that
$$ z_{\mu,\xi}(y+\xi)=\mu^{(2s-n)/2} \bar z \left(\frac{|y|^2}{\mu^2}\right),$$
which is also even in any of the variables~$y_i$.
As a consequence, using the change of variable~$x=y+\xi$ we obtain that,
for any~$i$, $j\in\{1,\dots,n\}$,
\begin{equation}\label{TT1}
\begin{split}
&\int_{\R^n} z_{\mu,\xi}^{p-1}(x)
\,\frac{\partial z_{\mu,\xi}}{\partial\xi_i}(x)
\,\frac{\partial z_{\mu,\xi}}{\partial\xi_j}(x)\,dx\\ &\qquad=
\int_{\R^n} z_{\mu,\xi}^{p-1}(y+\xi)
\,\frac{\partial z_{\mu,\xi}}{\partial\xi_i}(y+\xi)
\,\frac{\partial z_{\mu,\xi}}{\partial\xi_j}(y+\xi)\,dy
\\ &\qquad=
\int_{\R^n}
\mu^{(2s-n)(p+1)/2} \bar z^{p-1} \left(\frac{|y|^2}{\mu^2}\right)
(\bar z')^2\left( \frac{|y|^2}{\mu^2}\right)
\frac{4y_iy_j}{\mu^2}\,dy\\ &\qquad=
\left\{
\begin{matrix}
0 & {\mbox{ if $i\ne j$,}}\\
c_1 & {\mbox{ if $i= j$,}}
\end{matrix}
\right.
\end{split}\end{equation}
for some~$c_1>0$, which is bounded from zero uniformly.

Similarly, for any~$i\in\{1,\dots,n\}$,
\begin{equation}\label{TT2}
\begin{split}
& \int_{\R^n} z_{\mu,\xi}^{p-1}(x)
\,\frac{\partial z_{\mu,\xi}}{\partial\xi_i}(x)
\,\frac{\partial z_{\mu,\xi}}{\partial\mu}(x)\,dx = 0.
\end{split}\end{equation}
Finally, we observe that~$\bar z$ is positive and decreasing,
thus both~$\bar z$ and~$-\bar z'$ are positive: this says that the
right hand side of~\eqref{RHS99} is positive,
and indeed bounded from
zero uniformly. Hence we obtain that
\begin{equation}\label{TT3}
\int_{\R^n} z_{\mu,\xi}^{p-1}(x)\,
\left(\frac{\partial z_{\mu,\xi}}{\partial\mu}(x)\right)^2\,dx=c_2
\end{equation}
with~$c_2>0$ and bounded from zero uniformly.

Now, to make the notation
uniform, we take~$\zeta$, $\eta\in \{\xi_1,\dots,\xi_n,\mu\}$
and we consider the derivatives of~$z_{\mu,\xi}$
with respect to~$\zeta$ and~$\eta$.
Then we have that the quantity
$$\left\langle \frac{\partial z_{\mu,\xi}}{\partial\zeta},
\frac{\partial z_{\mu,\xi}}{\partial\eta}\right\rangle$$
is equal, up to dimensional constants, to
\begin{eqnarray*}
&&\int_{\R^n} (-\Delta)^{s/2}\frac{\partial z_{\mu,\xi}}{\partial\zeta}(x)
\, (-\Delta)^{s/2}\frac{\partial z_{\mu,\xi}}{\partial\eta}(x)\,dx \\
{\mbox{integrating by parts }}
&=& \int_{\R^n} (-\Delta)^{s}\frac{\partial z_{\mu,\xi}}{\partial\zeta}(x)
\,\frac{\partial z_{\mu,\xi}}{\partial\eta}(x)\,dx \\
{\mbox{exchanging the order of differentiation }}
&=& \int_{\R^n}
\frac{\partial }{\partial\zeta}
(-\Delta)^{s}z_{\mu,\xi}(x)
\,\frac{\partial z_{\mu,\xi}}{\partial\eta}(x)\,dx \\
{\mbox{using the equation }}&=& \int_{\R^n}
\frac{\partial }{\partial\zeta}
z_{\mu,\xi}^p(x)
\,\frac{\partial z_{\mu,\xi}}{\partial\eta}(x)\,dx \\
{\mbox{taking the derivative }}&=& p\int_{\R^n} z_{\mu,\xi}^{p-1}(x)
\frac{\partial z_{\mu,\xi} }{\partial\zeta}(x)
\,\frac{\partial z_{\mu,\xi}}{\partial\eta}(x)\,dx ,
\end{eqnarray*}
hence the desired result follows from~\eqref{TT1},
\eqref{TT2} and~\eqref{TT3}.
\end{proof}

\noindent Concerning the statement of Lemma~\ref{eigen},
we point out that the proof shows that~$\lambda_1=\dots=\lambda_n$
(while~$\lambda_{n+1}$ could be different), but in this paper we are not
taking advantage of this additional feature.

\subsection{Solving an auxiliary equation}
Keeping the notation introduced in the previous subsection,
the goal now is to solve an auxiliary equation
by means of the Implicit Function Theorem to obtain the following result.

\begin{lemma}\label{lemma0}
Let $z_{\mu,\xi}\in Z_0$. Then, for $\epsilon>0$ sufficiently small, there exists a unique
$w=w(\epsilon, z_{\mu,\xi})\in\left(T_{z_{\mu,\xi}}Z_0\right)^{\perp}$ such that
\begin{equation}\label{AUX}\begin{split}
&\iint_{\R^{2n}} \frac{\big( (z_{\mu,\xi}+w)(x)-(z_{\mu,\xi}+w)(y)\big)\,\big(
\varphi(x)-\varphi(y)\big)}{|x-y|^{n+2s}}\,dx\,dy
\\ &\qquad\,=\, \int_{\R^n} \Big( \epsilon h(x) \big( z_{\mu,\xi}(x)+w(x)\big)^q
+\big( z_{\mu,\xi}(x)+w(x)\big)^p \Big)
\,\varphi(x)\,dx,
\end{split}\end{equation}
for any $\varphi\in \left(T_{z_{\mu,\xi}}Z_0\right)^{\perp} \cap X^s$.

Moreover, the function $w$
is of class $C^1$ with respect to $\mu$ and $\xi$ and there exists a constant $C>0$ such that
\begin{equation}\label{est}
\|w\|_{X^s}\le C\,\epsilon, \ {\mbox{ and }} \
\lim_{\eps\to0}
\left\|\frac{\partial w}{\partial\mu}\right\|_{X^s}+
\left\|\frac{\partial w}{\partial\xi}\right\|_{X^s}=0.
\end{equation}
\end{lemma}

Indeed, recalling the definition of $U$ given in \eqref{defU}, we can set for any $u\in U$
\begin{equation}\label{A agg}
A_\epsilon(u):=\epsilon\,h\,u^q +u^p.
\end{equation}
We observe that $u=J(A_\epsilon(u))$ (where $J$ has been introduced in \eqref{Jpsi})
implies that $u$ solves (up to an unessential renormalizing constant
that we neglect for simplicity, recall the footnote on page~\pageref{P5})
$$ (-\Delta)^s u=A_\epsilon(u) \ {\mbox{ in }}\R^n, $$
thanks to Theorem \ref{THABC} (see in particular \eqref{EL-4}).
Moreover, we have that
\begin{equation}\label{poupty}
\|J(A_\epsilon(u))\|_{L^{\CRES}(\R^n)}<+\infty.
\end{equation}
Indeed, by \eqref{EL-1} in Theorem \ref{THABC} we get that there exists $C>0$ such that
\begin{equation}\label{bjfkbgugbher}
\|J(A_\epsilon(u))\|_{L^{\CRES}(\R^n)}\le C\|A_\epsilon(u)\|_{L^{\beta}(\R^n)},
\end{equation}
where $\beta=2n/(n+2s)$. Now, since $u\in L^{\CRES}(\R^n)$ and $p=(n+2s)/(n-2s)$,
we have that $u^p\in L^\beta(\R^n)$. This and the fact that $h$ is compactly supported
imply that $\|A_\epsilon(u)\|_{L^{\beta}(\R^n)}<+\infty$.
Therefore, from \eqref{bjfkbgugbher} we deduce \eqref{poupty}.

Analogously, making use of \eqref{EL-2} and \eqref{EL-3},
one sees that
$$ [J(A_\epsilon(u))]_{\dot H^s(\R^n)}+\|J(A_\epsilon(u))\|_{L^{\infty}(\R^n)}
<+\infty.$$
Hence, using Theorem~\ref{THABC}, we have that
if~$u\in U$ then~$J(A_\epsilon(u))\in X^s$.

Now, we use the notation $U\ni u=z_{\mu,\xi}+w$, with $
z_{\mu,\xi}\in Z_0$ and $w\in V$, and we
recall that we are identifying the manifold $Z_0$ defined in \eqref{critman}
with $(\mu_1,\mu_2)\times B_R\subset\R^{n+1}$. We define
\begin{equation}\label{defH}
H:(\mu_1,\mu_2)\times B_R\times V\times\R\times\R^{n+1}\rightarrow
X^s\times\R^{n+1}
\end{equation}
as $H=(H_1,H_2)$, with components
\begin{eqnarray*}
H_1(\mu,\xi,w,\epsilon,\alpha) &:=& z_{\mu,\xi}+w-J(A_\epsilon(z_{\mu,\xi}+w))
-\sum_{i=1}^{n+1}\alpha_i\,q_i, \\
H_2(\mu,\xi,w,\epsilon,\alpha) &:=& \left(\langle w,q_1\rangle,\ldots,\langle
w,q_{n+1}\rangle\right),
\end{eqnarray*}
where $q_i$ was defined in \eqref{qj}.

Our goal is to find $w=w(\epsilon, z_{\mu,\xi})$ (that we
also think as $w=w(\epsilon, \mu,\xi)$ with a slight abuse of notation) that
solves the equation $H(\mu,\xi,w,\epsilon,\alpha)=0$, that is
the system of equations
\begin{equation}\label{f6v7gb8idsfvvcwqw}
H_1(\mu,\xi,w,\epsilon,\alpha)=0=H_2(\mu,\xi,w,\epsilon,\alpha)
.\end{equation}
We notice that if $w$ satisfies \eqref{f6v7gb8idsfvvcwqw}
then~$w\in\left(T_{z_{\mu,\xi}}Z_0\right)^{\perp}$ and $z_{\mu,\xi}+w$ is a solution
of the auxiliary equation \eqref{AUX}. Indeed, $H_2(\mu,\xi,w,\epsilon,w)=0$ implies that
$$ \langle w,q_i\rangle=0 \ {\mbox{ for any }}i=1,\ldots,n+1,$$
which means that $w\in\left(T_{z_{\mu,\xi}}Z_0\right)^{\perp}$.
Moreover, $H_1(\mu,\xi,w,\epsilon,\alpha)=0$ gives that
$z_{\mu,\xi}+w-J(A_\epsilon(z_{\mu,\xi}+w))\in T_{z_{\mu,\xi}}Z_0$, and so
$$ \langle z_{\mu,\xi}+w-J(A_\epsilon(z_{\mu,\xi}+w)), \varphi\rangle =0$$
for any $\varphi\in\left(T_{z_{\mu,\xi}}Z_0\right)^{\perp}\cap X^s$. That is
\begin{equation}\begin{split}\label{jkwgoprt43th}
&\iint_{\R^{2n}}
\frac{\big( (z_{\mu,\xi}+w)(x)-(z_{\mu,\xi}+w)(y)\big)
\,\big( \varphi(x)-\varphi(y)\big)}{|x-y|^{n+2s}}
\,dx\,dy\\
&\qquad=\,
\iint_{\R^{2n}}
\frac{\big( J(A_\epsilon(z_{\mu,\xi}+w)) (x)-J(A_\epsilon(z_{\mu,\xi}+w))(y)\big)
\,\big( \varphi(x)-\varphi(y)\big)}{|x-y|^{n+2s}}
\,dx\,dy
\\&\qquad=\, \int_{\R^n} A_\epsilon(z_{\mu,\xi}+w)(x)\,\varphi(x)\,dx,
\end{split}\end{equation}
for any $\varphi\in\left(T_{z_{\mu,\xi}}Z_0\right)^{\perp}\cap X^s$,
thanks to \eqref{EL-4} in Theorem \ref{THABC}, which is \eqref{AUX}.

Therefore, to prove Lemma \ref{lemma0}, the strategy will be to apply the
Implicit Function Theorem to find a solution of the auxiliary
equation $H(\mu,\xi,w,\epsilon,\alpha)=0$.
Since we are working in the space~$X^s$,
it is not obvious that $H$ satisfies the hypotheses needed
to apply this theorem.
Indeed, the proofs of these requirements are very technically involved,
so we devote the next two subsections to study in detail the behavior of the operator $H$.

\subsubsection{Preliminary results on $H$}

Consider the operator defined in \eqref{defH}. First of all, we prove some continuity property.

\begin{lemma}\label{C1w}
$H$ is $C^1$ with respect to $w$.
\end{lemma}

\begin{proof}
We first notice that $H_2$ depends linearly on $w$,
and so it is $C^1$. Now we prove that $H_1$ is continuous in $X^s$.
Indeed, for any $w_1,w_2\in V$ we have that
$$ H_1(\mu,\xi,w_1,\epsilon,\alpha)-H_1(\mu,\xi,w_2,\epsilon,\alpha) =
w_1-w_2 - J(A_\epsilon(z_{\mu,\xi}+w_1))+ J(A_\epsilon(z_{\mu,\xi}+w_1)),$$
and therefore
\begin{equation}\begin{split}\label{beta73}
&\|H_1(\mu,\xi,w_1,\epsilon,\alpha)-H_1(\mu,\xi,w_2,\epsilon,\alpha)\|_{X^s} \\&\qquad \le
\|w_1-w_2\|_{X^s} +\|J(A_\epsilon(z_{\mu,\xi}+w_1))- J(A_\epsilon(z_{\mu,\xi}+w_2))\|_{X^s}.
\end{split}\end{equation}
By \eqref{EL-2} and \eqref{EL-3} of Theorem \ref{THABC}
and the fact that~$J$ is linear we deduce that
\begin{equation}\begin{split}\label{beta72}
&\|J(A_\epsilon(z_{\mu,\xi}+w_1))- J(A_\epsilon(z_{\mu,\xi}+w_2))\|_{X^s} \\& \le
C\left(\|A_\epsilon(z_{\mu,\xi}+w_1)- A_\epsilon(z_{\mu,\xi}+w_2)\|_{L^\infty(\R^n)}
+ \|A_\epsilon(z_{\mu,\xi}+w_1)- A_\epsilon(z_{\mu,\xi}+w_2)\|_{L^\beta(\R^n)}\right),
\end{split}\end{equation}
where $\beta=2n/(n+2s)$. Now from~\eqref{A agg} we deduce that
\begin{eqnarray*}
&& A_\epsilon(z_{\mu,\xi}+w_1)- A_\epsilon(z_{\mu,\xi}+w_2) \\
&=& \epsilon\,h\,\left[(z_{\mu,\xi}+w_1)^q-(z_{\mu,\xi}+w_2)^q \right]
+ (z_{\mu,\xi}+w_1)^p -(z_{\mu,\xi}+w_2)^p \\
&=& \epsilon q \,h\, (z_{\mu,\xi}+\tilde{w})^{q-1}(w_1-w_2) + p(z_{\mu,\xi}+\tilde{w})^{p-1}(w_1-w_2),
\end{eqnarray*}
for some $\tilde{w}$ on the segment joining $w_1$ and $w_2$
(in particular $\tilde{w}\in L^{\CRES}(\R^n)$ and $z_{\mu,\xi}+\tilde{w}$
satisfies \eqref{PP8}). Consequently,
\begin{equation}\label{beta71}
\|A_\epsilon(z_{\mu,\xi}+w_1)- A_\epsilon(z_{\mu,\xi}+w_2)\|_{L^\infty(\R^n)}
\le C\,\|w_1-w_2\|_{L^\infty(\R^n)}.
\end{equation}
Moreover, since $h$ has compact support, we have that
\begin{equation}\label{beta70}
\|\epsilon\,h\, (z_{\mu,\xi}+\tilde{w})^{q-1}(w_1-w_2) \|_{L^\beta(\R^n)}
\le C\,\|w_1-w_2\|_{L^\infty(\R^n)}.
\end{equation}
Finally, using H\"older inequality with exponent $\CRES/\beta=(n+2s)/(n-2s)$
and $\delta:=(n+2s)/4s$, we get
\begin{eqnarray*}
&&\|(z_{\mu,\xi}+\tilde{w})^{p-1}(w_1-w_2)\|_{L^\beta(\R^n)}^\beta \\
&&\qquad = \int_{\R^n}(z_{\mu,\xi}+\tilde{w})^{(p-1)\beta}(w_1-w_2)^\beta \\
&&\qquad \le \left(\int_{\R^n}(z_{\mu,\xi}+\tilde{w})^{(p-1)\beta\delta}\right)^{1/\delta}
\left(\int_{\R^n}(w_1-w_2)^{\CRES}\right)^{\beta/\CRES}\\
&&\qquad = \left(\int_{\R^n}(z_{\mu,\xi}+\tilde{w})^{\CRES}\right)^{1/\delta}
\left(\int_{\R^n}(w_1-w_2)^{\CRES}\right)^{\beta/\CRES} \\
&&\qquad \le C\, \|w_1-w_2\|_{L^{\CRES}(\R^n)}^\beta \\
&&\qquad \le C\, [w_1-w_2]_{\dot H^s(\R^n)}^\beta,
\end{eqnarray*}
up to renaming $C>0$, where we have used
Lemma \ref{L SOB}
in the last line.
Using this, \eqref{beta71} and \eqref{beta70}
into \eqref{beta72} we obtain that
$$ \|J(A_\epsilon(z_{\mu,\xi}+w_1))- J(A_\epsilon(z_{\mu,\xi}+w_2))\|_{X^s}\le C\,
\|w_1-w_2\|_{X^s}, $$
which together with \eqref{beta73} imply that
$$ \|H_1(\mu,\xi,w_1,\epsilon,\alpha)-H_1(\mu,\xi,w_2,\epsilon,\alpha)\|_{X^s} \le C\,
\|w_1-w_2\|_{X^s}, $$
up to renaming $C$. This shows the continuity of $H_1$ in $X^s$
with respect to $w$.

Now, in order to prove that $H_1$ is $C^1$, we observe that
\begin{equation}\begin{split}\label{der A}
\frac{\partial H_1}{\partial w}[v]=\,& v-J(A'_\epsilon(z_{\mu,\xi}+w)v)\\ =\,&
v- J\left(q\epsilon\,h\,(z_{\mu,\xi}+w)^{q-1}v+p(z_{\mu,\xi}+w)^{p-1}v\right).
\end{split}\end{equation}
To see this, we take $v\in V$ and $|t|<1$ and we compute
\begin{eqnarray*}
&& A_\epsilon(z_{\mu,\xi}+w+tv)-A_\epsilon(z_{\mu,\xi}+w)
\\&=& \epsilon\,h\,\left[(z_{\mu,\xi}+w+tv)^q-(z_{\mu,\xi}+w)^q\right]
+(z_{\mu,\xi}+w+tv)^p -(z_{\mu,\xi}+w)^p \\
&=& q\epsilon\,h\,(z_{\mu,\xi}+w)^{q-1}tv+p(z_{\mu,\xi}+w)^{p-1}tv +O(t^2),
\end{eqnarray*}
and so
$$ \lim_{t\to0}\frac{A_\epsilon(z_{\mu,\xi}+w+tv)-A_\epsilon(z_{\mu,\xi}+w)}{t} =
q\epsilon\,h\,(z_{\mu,\xi}+w)^{q-1}v+p(z_{\mu,\xi}+w)^{p-1}v. $$
From this and the fact that $J$ is linear we get that
\begin{eqnarray*}
\frac{\partial H_1}{\partial w}[v] &=&
\lim_{t\to0}\frac{1}{t}\left[tv +J\left(A_\epsilon(z_{\mu,\xi}+w+tv)-A_\epsilon(z_{\mu,\xi}+w)\right)
\right]\\
&=& v- J\left(q\epsilon\,h\,(z_{\mu,\xi}+w)^{q-1}v+p(z_{\mu,\xi}+w)^{p-1}v\right),
\end{eqnarray*}
which is \eqref{der A}. From \eqref{der A} we obtain that, for any $w_1,w_2\in V$,
\begin{equation}\begin{split}\label{uffa44}
& \left\|\frac{\partial H_1}{\partial w}(\mu,\xi,w_1,\epsilon,\alpha)
-\frac{\partial H_1}{\partial w}(\mu,\xi,w_2,\epsilon,\alpha)\right\|_{\mathcal{L}((X^s)^*,X^s)} \\
&\qquad = \sup_{\|v\|_{X^s}=1}
\left\|J(A'_\epsilon(z_{\mu,\xi}+w_1)v)-J(A'_\epsilon(z_{\mu,\xi}+w_2)v)\right\|_{X^s}.
\end{split}\end{equation}
Since $J$ is linear, by \eqref{EL-2} and \eqref{EL-3} in Theorem \ref{THABC}
we obtain that
\begin{equation}\begin{split}\label{uffa43}
& \left\|J(A'_\epsilon(z_{\mu,\xi}+w_1)v)-J(A'_\epsilon(z_{\mu,\xi}+w_2)v)\right\|_{X^s} \\
&\le C\left(\|A'_\epsilon(z_{\mu,\xi}+w_1)v-A'_\epsilon(z_{\mu,\xi}+w_2)v\|_{L^\infty(\R^n)}
+\|A'_\epsilon(z_{\mu,\xi}+w_1)v-A'_\epsilon(z_{\mu,\xi}+w_2)v\|_{L^\beta(\R^n)}\right),
\end{split}\end{equation}
where $\beta=2n/(n+2s)$. 

We have that
\begin{equation}\label{AONALLA}
\begin{split}
& A'_\epsilon(z_{\mu,\xi}+w_1)v-A'_\epsilon(z_{\mu,\xi}+w_2)v \\
&\qquad = q\,\epsilon\,h\,v\left[(z_{\mu,\xi}+w_1)^{q-1}-
(z_{\mu,\xi}+w_2)^{q-1}\right]
+ p\,v\,\left[(z_{\mu,\xi}+w_1)^{p-1}-(z_{\mu,\xi}+w_2)^{p-1} \right].
\end{split}\end{equation}
Now we distinguish two cases, either $p\le2$ or $p>2$.
If $p\le2$, we use \eqref{AONALLA} and we obtain that
\begin{equation}\begin{split}\label{uffa41}
& |A'_\epsilon(z_{\mu,\xi}+w_1)v-A'_\epsilon(z_{\mu,\xi}+w_2)v| \\
& \qquad \le
q|q-1|\,\epsilon\,|h|\,|v|\,|z_{\mu,\xi}+\tilde{w}|^{q-2}|w_1-w_2| + C\, |w_1-w_2|^{p-1}|v|,
\end{split}\end{equation}
for some $\tilde{w}$ on the segment joining $w_1$ and $w_2$.
Accordingly,
\begin{equation}\begin{split}\label{uffa42}
& \|A'_\epsilon(z_{\mu,\xi}+w_1)v-A'_\epsilon(z_{\mu,\xi}+w_2)v\|_{L^\infty(\R^n)}
\\&\qquad \le C\, \left(\|w_1-w_2\|_{L^\infty(\R^n)}+\|w_1-w_2\|_{L^\infty(\R^n)}^{p-1}\right),
\end{split}\end{equation}
since $z_{\mu,\xi}+\tilde{w}$ satisfies \eqref{PP8}.
Concerning the estimate for the $L^\beta$-norm, we observe that, since $h$
is compactly supported and $v\in L^\beta_{loc}(\R^n)$, we have
\begin{equation}\label{uffa40}
\|q|q-1|\,\epsilon\,|h|\,|v|\,|z_{\mu,\xi}+\tilde{w}|^{q-2}|w_1-w_2|\|_{L^\beta(\R^n)}
\le C\, \|w_1-w_2\|_{L^\infty(\R^n)}.
\end{equation}
Moreover, applying the
H\"older inequality with exponents $\frac{\CRES}{(p-1)\beta}=\frac{n+2s}{4s}$
and $p$ we obtain that
\begin{eqnarray*}
\||w_1-w_2|^{p-1}|v|\|_{L^\beta(\R^n)}^\beta &=&\int_{\R^n}|w_1-w_2|^{(p-1)\beta}|v|^\beta \\
& \le & \left(\int_{\R^n}|w_1-w_2|^{\CRES}\right)^{4s/(n+2s)}
\left(\int_{\R^n}|v|^{p\beta}\right)^{1/p} \\
& =&  \|w_1-w_2\|_{L^{\CRES}(\R^n)}^{8ns/[(n+2s)(n-2s)]}
\left(\int_{\R^n}|v|^{\CRES}\right)^{1/p} \\
&\le & C\, \|w_1-w_2\|_{L^{\CRES}(\R^n)}^{8ns/[(n+2s)(n-2s)]},
\end{eqnarray*}
for a suitable positive constant $C$. Hence, by
Lemma \ref{L SOB},
we have that
$$ \||w_1-w_2|^{p-1}|v|\|_{L^\beta(\R^n)}\le C\, \|w_1-w_2\|_{L^{\CRES}(\R^n)}^{4s/(n-2s)}
\le C\, [w_1-w_2]_{\dot H^s(\R^n)}^{4s/(n-2s)}, $$
up to relabelling $C$. This, \eqref{uffa40} and \eqref{uffa41} imply that
$$ \|A'_\epsilon(z_{\mu,\xi}+w_1)v-A'_\epsilon(z_{\mu,\xi}+w_2)v\|_{L^\beta(\R^n)}\le
C\, \Big( \|w_1-w_2\|_{X^s} + \|w_1-w_2\|_{X^s}^{4s/(n-2s)} \Big).$$
Putting together this, \eqref{uffa42}, \eqref{uffa43} and \eqref{uffa44},
we obtain that $\frac{\partial H_1}{\partial w}$ is continuous
with respect to $w$ in $X^s$. This implies that $H_1$ is $C^1$ with respect to $w$,
and concludes the proof when $p\le2$.

If instead $p\ge2$, we deduce from \eqref{AONALLA} that
\begin{eqnarray*}
&& |A'_\epsilon(z_{\mu,\xi}+w_1)v-A'_\epsilon(z_{\mu,\xi}+w_2)v| \\
&& \qquad \le
q|q-1|\,\epsilon\,|h|\,|v|\,|z_{\mu,\xi}+\tilde{w}|^{q-2}|w_1-w_2| + C\,
|v|\,(z_{\mu,\xi}+|w_1|+|w_2|)^{p-2}\,|w_1-w_2|.
\end{eqnarray*}
The first term in the right hand side of this inequality can be treated
as in the case $p\le2$, so we focus on the latter term. To this aim,
we first bound $|v|\,(z_{\mu,\xi}+|w_1|+|w_2|)^{p-2}\,|w_1-w_2|$
in $L^\infty(\R^n)$ by $C\,\|w_1-w_2\|_{L^\infty(\R^n)}$,
which assures the desired bound in $L^\infty(\R^n)$.
Hence, we are left with the estimate for the norm in $L^\beta(\R^n)$ of this term.
For this, we use the H\"older
inequality with exponents~$\frac{2^*_s \,p}{2^*_s \,(p-1)-\beta p}=\frac{p}{p-2}$,
$\frac{2^*_s}{\beta}=p$ and $\frac{2^*_s}{\beta}=p$,
and we find that
\begin{eqnarray*}
&& \big\| (z_{\mu,\xi}+|w_1|+|w_2|)^{p-2}\,|w_1-w_2|\,|v|\,
\big\|_{L^\beta(\R^n)}^\beta \\ &&\qquad\le
\| z_{\mu,\xi}+|w_1|+|w_2| \,\big\|_{L^{2^*_s}(\R^n)}^{ \frac{2^*_s \,(p-1)-\beta p}{p} }
\, \|w_1-w_2\|_{L^{2^*_s}(\R^n)}^{\beta}\,
\|v\|_{L^{2^*_s}(\R^n)}^{\beta}\\ &&\qquad\le C\,
\|w_1-w_2\|_{L^{2^*_s}(\R^n)}^{ \beta}.\end{eqnarray*}
As in the case $p\le2$, this estimate
implies that $H_1$ is $C^1$ with respect to $w\in X^s$
and so it concludes the proof also when $p>2$.
\end{proof}

Let us study now some properties of the derivative of $H$.
In particular, consider first the operator
\begin{equation}\label{DE T}
Tv:=\frac{\partial H_1}{\partial w}(\mu,\xi,0,0,0)[v]=
v-J(A_0'(z_{\mu,\xi})v).\end{equation}
This definition is well posed, as next result points out:

\begin{lemma}\label{po0}
$T$ is a bounded operator from~$\dot H^s(\R^n)$ to~$\dot H^s(\R^n)$.
\end{lemma}

\begin{proof} Let~$\psi:=A'_0(z_{\mu,\xi}) v=p z_{\mu,\xi}^{p-1} v$.
{F}rom~\eqref{EL-2}, we know that
$$ [J(A'_0(z_{\mu,\xi}) v)]_{\dot{H}^s(\R^n)}
=[J\psi]_{\dot{H}^s(\R^n)} \le C\,\|\psi\|_{L^\beta(\R^n)}
= Cp\,\|z_{\mu,\xi}^{p-1} v\|_{L^\beta(\R^n)},$$
with~$\beta=2n/(n+2s)$.
On the other hand, using the H\"older inequality with exponents~$\CRES/\beta$
and~$(n+2s)/4s$ we can bound the quantity~$\|z_{\mu,\xi}^{p-1} v\|_{L^\beta(\R^n)}$
with~$C\,\|v\|_{L^{\CRES}(\R^n)}$ and thus by~$C\,[v]_{\dot{H}^s(\R^n)}$, thanks to
the Sobolev inequality. This gives that
$$ [J(A'_0(z_{\mu,\xi}) v)]_{\dot{H}^s(\R^n)}
\le C\,[v]_{\dot{H}^s(\R^n)},$$
which implies the desired result.
\end{proof}

It is important to remark that~$T$ is also a linear operator over~$X^s$.
Of course, since~$X^s$ is a subset of~$\dot H^s(\R^n)$, the restriction operator, that
we still denote by~$T$, maps~$X^s$ continuously to~$\dot H^s(\R^n)$.
What is relevant for us is that it also maps~$X^s$ continuously to~$X^s$,
as next result explicitly states:

\begin{lemma}\label{po00}
$T$ is a bounded operator from~$X^s$ to~$X^s$.
\end{lemma}

\begin{proof} Same as the one of Lemma~\ref{po0},
using~\eqref{EL-3}
in addition to~\eqref{EL-2}.
\end{proof}

As a matter of fact, $T$ enjoys further compactness properties, as observed
in the next result:

\begin{prop}\label{fredholm}
$T$ is a Fredholm operator over~$\dot{H}^s(\R^n)$. More explicitly,
if we set~${K}v:=-J(A_0'(z_{\mu,\xi})v)$,
we have that~$T=Id_{\dot{H}^s(\R^n)}+{K}$, and
${K}:\dot{H}^s(\mathbb{R}^n)\rightarrow\dot{H}^s(\mathbb{R}^n)$ is a compact operator
over~$\dot{H}^s(\R^n)$.
\end{prop}

\begin{proof} We already know from Lemma~\ref{po0} that~${K}$ is a bounded
operator over~$\dot{H}^s(\mathbb{R}^n)$. Now, let $\{v_k\}_{k\in\mathbb{N}}$ be a sequence such that
\begin{equation}\label{poi}
[v_k]_{\dot{H}^s(\mathbb{R}^n)}\leq 1.\end{equation} To prove compactness, we need to see that
\begin{equation}\label{CAU}{\mbox{
$\{{K}v_k\}_{k\in\mathbb{N}}$ contains a Cauchy subsequence in~$\dot H^s(\R^n)$.
}}\end{equation}For this, we fix~$\varepsilon>0$ and we
exploit~\eqref{EL-2} of Theorem
\ref{THABC} to obtain that
\begin{equation}\label{pio}\begin{split}
&[{K}v_l-{K}v_m]_{\dot{H}^s(\mathbb{R}^n)}\\&\quad=
[J(A_0'(z_{\mu,\xi})(v_l-v_m))]_{\dot{H}^s(\mathbb{R}^n)}\\ &\quad\leq
C\|A_0'(z_{\mu,\xi})(v_l-v_m)\|_{L^\beta(\mathbb{R}^n)}\\
&\quad=C(\|A_0'(z_{\mu,\xi})(v_l-v_m)\|_{L^\beta(B_R)}+\|A_0'(z_{\mu,\xi})(v_l-v_m)\|_{L^\beta(\mathbb{R}^n\setminus
B_R)}), \end{split}\end{equation} where $\beta:=\frac{2n}{n+2s}$, $R>0$, and $B_R:=\{x\in\mathbb{R}^n:
|x|<R\}$.

Thus we notice that, for a fixed~$R>0$, the quantity~$\|v_k\|_{L^2(B_R)}$
is bounded by~$\|v_k\|_{L^{\CRES}(B_R)}$, by H\"older inequality, and the latter quantity
is in turn
bounded by~$[v_k]_{\dot{H}^s(\mathbb{R}^n)}$, by Sobolev inequality. These
observations and~\eqref{poi} imply that
$$\|v_k\|_{W^{s,2}(B_R)}\leq C_R,$$
for some~$C_R>0$ that does not depend on~$k$.
Moreover, the space $W^{s,2}(B_R)$ is compactly embedded in $L^\beta(B_R)$
(see Corollary 7.2 in \cite{DPV} and recall that~$\beta\in(1,\CRES)$).
This implies that~$v_k$ contains a Cauchy subsequence in~$L^\beta(B_R)$ and so,
up to a subsequence, if~$l$ and~$m$ are sufficiently large (say~$l$, $m\ge N(R,\varepsilon)$,
for some large~$N(R,\varepsilon)$) we have that
$$ \|v_l-v_m\|_{L^\beta(B_R)} \le \varepsilon.$$
Notice also that
$$A_0'(z_{\mu,\xi})=pz_{\mu,\xi}^{\frac{4s}{n-2s}}\in L^\infty(\mathbb{R}^n),$$
therefore
\begin{equation}\label{poi0} \|A_0'(z_{\mu,\xi})(v_l-v_m)\|_{L^\beta(B_R)}\leq
\|A_0'(z_{\mu,\xi})\|_{L^\infty(\mathbb{R}^n)}\|v_l-v_m\|_{L^\beta(B_R)}\le C\varepsilon\end{equation}
as long as~$l$, $m\ge N(R,\varepsilon)$.

On the other hand, applying H\"older and Sobolev inequalities, and recalling~\eqref{poi}
once again,
\begin{eqnarray*}
&&\|A_0'(z_{\mu,\xi})(v_l-v_m)\|_{L^\beta(\mathbb{R}^n\setminus B_R)}\\
&\leq&\left(\int_{\mathbb{R}^n\setminus B_R}{(v_l-v_m)^{\CRES}\,dx}\right)^{1/\CRES}
\left(\int_{\mathbb{R}^n\setminus B_R}{(pz_{\mu,\xi}^{\frac{4s}{n-2s}})^{\frac{n}{2s}}\,dx}\right)^{2s/n}\\
&\leq&C\|v_l-v_m\|_{L^{\CRES}(\mathbb{R}^n)}
\left(\int_{\mathbb{R}^n\setminus B_{\frac{R-|\xi|}{\mu}}}{\frac{1}{|y|^{2n}}\,dy}\right)^{2s/n}\\
&\leq& C [v_l-v_m]_{\dot{H}^s(\mathbb{R}^n)}R^{-n}\\
&\leq& C R^{-n},
\end{eqnarray*}
with $C>0$ possibly different from line to line, but
independent of $R$, $l$ and $m$.
Thus, we insert this and~\eqref{poi0} into~\eqref{pio} and we deduce that
$$ [{K}v_l-{K}v_m]_{\dot{H}^s(\mathbb{R}^n)}\le C(\varepsilon+R^{-n}),$$
provided that~$l$, $m\ge N(R,\varepsilon)$, possibly up to a subsequence.
In particular, we can choose~$R$ depending on~$\varepsilon$, for instance~$R:=\varepsilon^{-1/n}$,
and define~$N_\varepsilon:=N(\varepsilon^{-1/n},\varepsilon)$. So we obtain that,
for~$l$, $m\ge N_\varepsilon$, the quantity~$[{K}v_l-{K}v_m]_{\dot{H}^s(\mathbb{R}^n)}$
is bounded by a constant times~$\varepsilon$.
This establishes~\eqref{CAU}.
\end{proof}

Finally, for any~$(v,\beta)\in \dot H^s(\R^n)\times\R^{n+1}$ we define the linear operator
\begin{equation}\label{OP}
{\mathcal{T}}(v,\beta):=\left(
Tv-\sum_{i=1}^{n+1}\beta_i q_i, \langle v,q_1\rangle,\dots,
\langle v,q_{n+1}\rangle\right),\end{equation}
with $T$ defined in \eqref{DE T}. The interest of such operator for us is that
\begin{equation}\label{OP secondo}
\frac{\partial H}{\partial (w,\alpha)}(\mu,\xi,0,0,0)[v,\beta]=
{\mathcal{T}}(v,\beta).
\end{equation}
We have:
\begin{prop}\label{OP P}
${\mathcal{T}}$
is a bounded operator from~$\dot H^s(\R^n)\times\R^{n+1}$ to~$\dot H^s(\R^n)\times\R^{n+1}$,
and from~$X^s\times\R^{n+1}$ to~$X^s\times\R^{n+1}$.

Furthermore,
${\mathcal{T}}$ is a Fredholm operator over~$\dot{H}^s(\R^n)\times\R^{n+1}$. More explicitly,
it can be written as the identity plus a compact operator
over~$\dot{H}^s(\R^n)\times\R^{n+1}$.
\end{prop}

\begin{proof} Let
$$ S(v,\beta):=\left(-\sum_{i=1}^{n+1}\beta_i q_i, \langle v,q_1\rangle,\dots,
\langle v,q_{n+1}\rangle\right).$$
Let also~$\|\cdot\|$ be either~$\|\cdot\|_{\dot{H}^s(\R^n)}$ or~$\|\cdot\|_{X^s}$.
We have that
\begin{eqnarray*}
\|S(v,\beta)\| &\le&
\sum_{i=1}^{n+1}|\beta_i|\, \|q_i\|+\sum_{i=1}^{n+1}
\|v\|_{\dot{H}^s(\R^n)} \|q_i\|_{\dot{H}^s(\R^n)} \\
&\le& C\Big( |\beta| + \|v\|_{\dot{H}^s(\R^n)}\Big)\\
&\le& C\Big( |\beta| + \|v\|\Big).
\end{eqnarray*}
This shows that~$S$ is a bounded operator from~$\dot H^s(\R^n)\times\R^{n+1}$ to~$\dot H^s(\R^n)\times\R^{n+1}$,
and from~$X^s\times\R^{n+1}$ to~$X^s\times\R^{n+1}$. Then, noticing that~${\mathcal{T}}=(T,0)+S$
and recalling Lemmata~\ref{po0}
and~\ref{po00}, we obtain that also~${\mathcal{T}}$ is a bounded
operator from~$\dot H^s(\R^n)\times\R^{n+1}$ to~$\dot H^s(\R^n)\times\R^{n+1}$,
and from~$X^s\times\R^{n+1}$ to~$X^s\times\R^{n+1}$.

Now we show that it is Fredholm over~$\dot H^s(\R^n)\times\R^{n+1}$.
For this, we set
$$ {\mathcal{K}}(v,\beta):=\left( {K}v
-\sum_{i=1}^{n+1}\beta_i q_i, \langle v,q_1\rangle-\beta_1,\dots,
\langle v,q_{n+1}\rangle-\beta_{n+1}
\right),$$
where~${K}$ is the operator in Proposition~\ref{fredholm}.
Notice that~${\mathcal{T}}=Id_{\dot{H}^s(\R^n)\times\R^{n+1}}+{\mathcal{K}}$,
so our goal is to show that~${\mathcal{K}}$ is compact over~$\dot H^s(\R^n)\times\R^{n+1}$.
For this, we take a sequence~$(v_k,\beta_k)\in\dot H^s(\R^n)\times\R^{n+1}$
with~$\|v_k\|_{\dot H^s(\R^n)}+\|\beta_k\|_{\R^{n+1}}\le 1$
and we want to find a Cauchy subsequence in~$\dot H^s(\R^n)\times\R^{n+1}$.
To this goal, we use Proposition~\ref{fredholm} to obtain
a subsequence (still denoted by~$v_k$) such that~${K}v_k$
is Cauchy in~$\dot H^s(\R^n)$. Also, again up to subsequences,
$v_k$ is weakly convergent in~$\dot H^s(\R^n)$,
therefore~$\langle v_k,q_1\rangle$ is Cauchy (and the same holds
for~$\langle v_k,q_2\rangle,\dots,\langle v_k,q_{n+1}\rangle$).
Finally, since~$\R^{n+1}$ is finite dimensional,
up to subsequence we can assume that also~$\beta_k$ is Cauchy.
Thanks to these considerations, and writing~$\beta_k=(\beta_{k,1},\dots,\beta_{k,n+1})\in\R^{n+1}$, we have that
\begin{eqnarray*}
&& \big\|  {\mathcal{K}}(v_k,\beta_k)- {\mathcal{K}}(v_m,\beta_m)\big\|_{\dot H^s(\R^n)\times\R^{n+1}}\\
&\le& \|{K}v_k-{K}v_m\|_{\dot H^s(\R^n)}
+\sum_{i=1}^{n+1}|\beta_{k,i}-\beta_{m,i}|\,\| q_i\|_{\dot H^s(\R^n)}
+\sum_{i=1}^{n+1} |\langle v_k-v_m,\,q_i\rangle| \\
&\le& C\,\left(
\|{K}v_k-{K}v_m\|_{\dot H^s(\R^n)}+
\|\beta_k-\beta_m\|_{\R^{n+1}} +
\sum_{i=1}^{n+1} |\langle v_k-v_m,\,q_i\rangle|
\right)
\\&\le&\varepsilon,
\end{eqnarray*}
provided that~$k$ and~$m$ are large enough. This shows that~$(v_k,\beta_k)$
is Cauchy, as desired.
\end{proof}

\subsubsection{Invertibility issues}

Now we discuss the invertibility of the operator~${\mathcal{T}}$
that was introduced in \eqref{OP}.
Notice that there is a subtle point here. Indeed,
the operator~${\mathcal{T}}$ can be seen
as acting over~$\dot H^s(\R^n)\times\R^{n+1}$
or over~$X^s\times\R^{n+1}$ (see Proposition~\ref{OP P}).
On the one hand, the invertibility over~$\dot H^s(\R^n)\times\R^{n+1}$
should be expected to be easier, since the operator is Fredholm there
(see the last claim in Proposition~\ref{OP P}).
On the other hand,
since we want to obtain strong
pointwise estimates to keep control of the possible singularities of our functional,
it is crucial for us to invert the operator
in a space that controls the functions
uniformly, namely~$X^s\times\R^{n+1}$.
So our strategy will be the following:
first we invert the operator in~$\dot H^s(\R^n)\times\R^{n+1}$
(this will be accomplished using the Fredholm property in Proposition~\ref{OP P},
the regularity theory in Theorem~\ref{LinfResult}
and a nondegeneracy result in~\cite{DDS}).
Then we will deduce from this information and a further regularity theory
that~${\mathcal{T}}$ is actually invertible also in~$X^s\times\R^{n+1}$.

The details of the argument go as follows. First, we recall the standard
definition of invertibility:

\begin{defn}
Let $X,Y$ Banach spaces, and let $S:X\rightarrow Y$ be a linear bounded operator. We say
that $S$ is invertible
(and we write~$S\in Inv(X,Y)$) if there exists a
linear bounded operator $\tilde{S}:Y\rightarrow X$ such that
$$S\tilde{S}=Id_Y,\quad \tilde{S}S=Id_X.$$
\end{defn}

Then, we show that~${\mathcal{T}}$ is invertible
in~$\dot H^s(\R^n)\times\R^{n+1}$:

\begin{prop}\label{INV 1}
${\mathcal{T}}\in Inv(\dot{H}^s(\mathbb{R}^n)\times\R^{n+1},\,\dot{H}^s(\mathbb{R}^n)\times\R^{n+1})$.
\end{prop}

\begin{proof} By Proposition~\ref{OP P} and the theory of Fredholm operators
(see e.g. \cite{B}, pages 168-169, for a very brief summary, and Chapter IV, Section 5, of \cite{K}, or \cite{S}, for a detailed analysis), it is enough to show that~${\mathcal{T}}$
is injective over~$\dot{H}^s(\mathbb{R}^n)\times\R^{n+1}$. For this, let us take~$(v,\beta)\in \dot{H}^s(\mathbb{R}^n)\times\R^{n+1}$ such that~${\mathcal{T}}(v,\beta)=0$, that is, by~\eqref{OP},
\begin{equation}\label{trew}\begin{split}
& Tv=\sum_{i=1}^{n+1}\beta_i q_i,\\
&\langle v,q_1\rangle=\dots=
\langle v,q_{n+1}\rangle=0.
\end{split}\end{equation}
Fixed~$j\in\{1,\dots,n+1\}$,
using~\eqref{DE T}, \eqref{EL-4} and~\eqref{EQqj}, we observe that
\begin{equation}\label{spi}\begin{split}
\langle Tv, q_j\rangle\;&=
\langle v-pJ(z_{\mu,\xi}^{p-1}v),\,q_j\rangle \\
&= \langle v,\,q_j\rangle-p\int_{\R^n}(-\Delta)^s J(z_{\mu,\xi}^{p-1}v)\,q_j
\\ &=
\langle v,\,q_j\rangle-p\int_{\R^n} z_{\mu,\xi}^{p-1}v\,q_j \\
&=
\langle v,\,q_j\rangle-\langle v,\,q_j\rangle
\\ &=0.\end{split}\end{equation}
This, \eqref{trew} and Lemma~\ref{eigen} give that
$$ 0=\langle Tv,q_j\rangle=\sum_{i=1}^{n+1}\beta_i \langle q_i,q_j\rangle
=\lambda_j\beta_j,$$
and so
\begin{equation}\label{ewfd4}
{\mbox{$\beta_j=0$ for every $j\in\{1,\dots,n+1\}$.}}\end{equation}
Therefore, $v\in\dot H^s(\R^n)$ is a weak solution of~$Tv=0$,
that is, by~\eqref{DE T}
and~\eqref{EL-4}, the equation~$(-\Delta)^s v=pz_{\mu,\xi}^{p-1}v$.
Accordingly, by Theorem~\ref{LinfResult}, we obtain that~$v\in L^\infty(\R^n)$.

Thanks to this, we can apply the
nondegeneracy result in~\cite{DDS}, that gives that~$v$ must be a linear combination of~$q_1,\dots,q_{n+1}$.
So we write
\begin{equation}\label{saddgf}
v=\sum_{i=1}^{n+1} c_i q_i\end{equation}
for some~$c_i\in\R$, we recall~\eqref{trew} and once again
Lemma~\ref{eigen}, and we compute
$$ 0=\langle v,q_j\rangle=\sum_{i=1}^{n+1} c_i \langle q_i,q_j\rangle=
c_j\lambda_j,$$
that gives~$c_j=0$ for every~$j\in\{1,\dots,n+1\}$. By plugging this
information into~\eqref{saddgf}, we conclude that~$v=0$.
This and~\eqref{ewfd4}
give that~$(v,\beta)=0$ and so~${\mathcal{T}}$ is injective on~$\dot H^s(\R^n)\times\R^{n+1}$.
\end{proof}

Next, we aim to prove that~${\mathcal{T}}\in Inv(
X^s\times\R^{n+1},\,X^s\times\R^{n+1})$.
For this scope, we need an improved regularity theory result, which goes as follows:

\begin{lemma}\label{qwrhgsa44LEMMA}
Let~$C_o>0$.
For any~$u\in X^s$, $(\alpha,\beta)\in\R^{n+1}\times\R^{n+1}$ and
any~$\psi\in\dot{H}^s(\R^n)$ which is a weak solution of
\begin{equation}\label{7sdvfbrt5} (-\Delta)^s \psi =p\sum_{i=1}^{n+1} \alpha_i z_{\mu,\xi}^{p-1}q_i
+pz_{\mu,\xi}^{p-1}\psi+pz_{\mu,\xi}^{p-1}u\end{equation}
with
\begin{equation}\label{qwrhgsa44}
[\psi]_{\dot H^s(\R^n)} \le C_o \Big(
\|u\|_{X^s}+\|\beta\|_{\R^{n+1}}\Big),\end{equation}
we have that~$\psi\in L^\infty(\R^n)$ and
\begin{equation}\label{67}
\|\psi\|_{L^\infty(\R^n)}\le C\Big(
\|u\|_{X^s}+\|\alpha\|_{\R^{n+1}}+\|\beta\|_{\R^{n+1}}\Big)
\end{equation}
for some~$C>0$.
\end{lemma}

\begin{proof} The core of the proof is that the equation
is linear in the triplet~$(\psi,u,\alpha)$, so we get
the desired result by a careful scaling argument.
The rigorous argument goes as follows.
First, we use Theorem~\ref{LinfResult} to get that~$\psi\in L^\infty(\R^n)$,
so we focus on the proof of~\eqref{67}.
Suppose, by contradiction, that~\eqref{67} is false.
Then, for any~$k$ there exists
a quadruplet~$(\psi_k,u_k,\alpha_k,\beta_k)\in \dot{H}^s(\R^n)\times X^s\times\R^{n+1}\times\R^{n+1}$
such that
\begin{equation}\label{68} (-\Delta)^s \psi_k
=p\sum_{i=1}^{n+1} \alpha_{k,i} z_{\mu,\xi}^{p-1}q_i
+pz_{\mu,\xi}^{p-1}\psi_k+pz_{\mu,\xi}^{p-1}u_k,\end{equation}
\begin{equation}\label{69}
\|\psi_k\|_{L^\infty(\R^n)}>k\,\Big(
\|u_k\|_{X^s}+\|\alpha_k\|_{\R^{n+1}}+\|\beta_k\|_{\R^{n+1}}\Big)
\end{equation}
and
\begin{equation}\label{890} [\psi_k]_{\dot H^s(\R^n)} \le C_o \Big(
\|u_k\|_{X^s}+\|\beta_k\|_{\R^{n+1}}\Big).\end{equation}
We remark that~$\|\psi_k\|_{L^\infty(\R^n)}<+\infty$,
since~$\psi_k\in L^\infty(\R^n)$, and~$\|\psi_k\|_{L^\infty(\R^n)}>0$,
due to~\eqref{69}. Thus, we can define
\begin{eqnarray*}&& \tilde \psi_k:=\frac{\psi_k}{\|\psi_k\|_{L^\infty(\R^n)}},
\qquad \tilde u_k:=\frac{u_k}{\|\psi_k\|_{L^\infty(\R^n)}},\\&&\qquad
\qquad\tilde \alpha_k:=\frac{\alpha_k}{\|\psi_k\|_{L^\infty(\R^n)}}
\qquad{\mbox{ and }}
\qquad\tilde \beta_k:=\frac{\beta_k}{\|\psi_k\|_{L^\infty(\R^n)}}.\end{eqnarray*}
Notice that
\begin{equation}\label{70}\begin{split}
& \|\tilde\psi_k\|_{L^\infty(\R^n)}=1\\
{\mbox{and }} \quad&
\|\tilde u_k\|_{X^s}+\|\tilde\alpha_k\|_{\R^{n+1}}
+\|\tilde\beta_k\|_{\R^{n+1}}
=\frac{\|u_k\|_{X^s}+\|\alpha_k\|_{\R^{n+1}}+\|\beta_k\|_{\R^{n+1}}
}{\|\psi_k\|_{L^\infty(\R^n)}}\le\frac{1}{k},
\end{split}\end{equation}
thanks to~\eqref{69}.

Also, by linearity, equation~\eqref{68} becomes
\begin{equation*}
(-\Delta)^s \tilde\psi_k
=p\sum_{i=1}^{n+1} \tilde\alpha_{k,i} z_{\mu,\xi}^{p-1}q_i
+pz_{\mu,\xi}^{p-1}\tilde\psi_k+pz_{\mu,\xi}^{p-1}\tilde u_k.\end{equation*}
The right hand side of this equation is bounded uniformly in~$L^\infty(\R^n)$,
thanks to~\eqref{70} and the fact that~$z_{\mu,\xi}\in L^\infty(\mathbb{R}^n)$.

Thus, by Proposition 5 in \cite{SV}, we
know that for every $x\in\mathbb{R}^n$, there exists a constant $C>0$
and $a\in(0,1)$ such that
\begin{equation*}
\|\tilde\psi_k\|_{C^a(B_{1/4}(x))}\leq C. \end{equation*}
We remark that~$C$ and~$a$ are independent of $k$ and $x$, therefore
\begin{equation}\label{CalphaBound}
\|\tilde\psi_k\|_{C^a(\mathbb{R}^n)}\leq C. \end{equation}
{F}rom~\eqref{70}, we know that there exists a point~$x_k\in \R^n$
such that~$\tilde\psi_k(x_k)\ge 1/2$. By~\eqref{CalphaBound},
there exists~$\rho>0$, which is independent of~$k$, such that~$\tilde\psi_k
\ge 1/4$ in~$B_\rho(x_k)$. As a consequence,
$$\|\tilde\psi_k\|_{L^{\CRES}(\mathbb{R}^n)}\geq
\left(\int_{B_\rho(x_k)}{\left(\frac{1}{4}\right)^{\CRES}\,dx}\right)^{1/\CRES}
\geq c_o,$$ with $c_o>0$ independent of $k$. Thus, by Sobolev inequality,
\begin{equation}\label{891}
[\tilde\psi_k]_{\dot H^s(\mathbb{R}^n)}\ge c_o,
\end{equation}
up to renaming~$c_o$.
On the other hand, by~\eqref{890} and~\eqref{69}, we have that
\begin{eqnarray*}
[\tilde\psi_k]_{\dot H^s(\mathbb{R}^n)} =
\frac{[\psi_k]_{\dot H^s(\R^n)}}{\|\psi_k\|_{L^\infty(\R^n)}} \le
\frac{ C_o \Big(
\|u_k\|_{X^s}+\|\beta_k\|_{\R^{n+1}}\Big)}{ \|\psi_k\|_{L^\infty(\R^n)} }
\le\frac{C_o}{k}.
\end{eqnarray*}
This is in contradiction with~\eqref{891} when~$k$ is large, and therefore
the desired result is established.
\end{proof}

Finally, we show that~${\mathcal{T}}$ is invertible in~$X^s(\R^n)\times\R^{n+1}$:

\begin{prop}\label{invertibility}
${\mathcal{T}}\in Inv(X^s\times\R^{n+1},\,X^s\times\R^{n+1})$.
\end{prop}

\begin{proof}
By Proposition~\ref{INV 1}, we know that
${\mathcal{T}}\in Inv(\dot{H}^s(\mathbb{R}^n)\times\R^{n+1},
\dot{H}^s(\mathbb{R}^n)\times\R^{n+1})$. Therefore,
there exists an operator
$$\tilde{{\mathcal{T}}}:\dot{H}^s(\mathbb{R}^n)\times\R^{n+1}
\rightarrow \dot{H}^s(\mathbb{R}^n)\times\R^{n+1}$$
that is linear and bounded
and such that~${\mathcal{T}}\tilde{{\mathcal{T}}}=
\tilde{{\mathcal{T}}}{\mathcal{T}}=Id_{\dot{H}^s(\R^n)\times\R^{n+1}}$.
The boundedness of~$\tilde{\mathcal{T}}$ as an operator acting over~$\dot{H}^s(\mathbb{R}^n)\times\R^{n+1}$
can be explicitly written as
\begin{equation}\label{ass0}
\|\tilde{\mathcal{T}}(u,\beta)\|_{\dot{H}^s(\mathbb{R}^n)\times\R^{n+1}}
\le C\,
\|(u,\beta)\|_{\dot{H}^s(\mathbb{R}^n)\times\R^{n+1}}.\end{equation}
Now, since $X^s$ is a subset of~$\dot{H}^s(\mathbb{R}^n)$,
we can consider the restriction operator of~$\tilde{{\mathcal{T}}}$ acting
on~$X^s\times\R^{n+1}$ (this restriction operator
will be denoted by~$\tilde{{\mathcal{T}}}$ as well).
We observe that, for any~$u\in X^s$, we have that~$u\in\dot{H}^s(\R^n)$,
therefore, for any~$\beta\in\R^{n+1}$,
$$ {\mathcal{T}}\tilde{{\mathcal{T}}} (u,\beta)=Id_{\dot{H}^s(\R^n)\times\R^{n+1}}
(u,\beta)=(u,\beta).$$
Furthermore, if~$u\in X^s$ and~$\beta\in\R^{n+1}$,
then~${\mathcal{T}}(u,\beta)\in X^s\times\R^{n+1}$, due to Proposition~\ref{OP P}.
Hence
the restriction of~$\tilde{{\mathcal{T}}}$ over~$X^s\times\R^{n+1}$
may act on~${\mathcal{T}} (u,\beta)$, for any~$(u,\beta)\in
X^s\times\R^{n+1}$, and we obtain that
$$ \tilde{{\mathcal{T}}}{\mathcal{T}} (u,\beta)=Id_{\dot{H}^s(\R^n)\times\R^{n+1}}(u,\beta)=(u,\beta).$$
It remains to prove that
\begin{equation}\label{goal1}
\|\tilde{\mathcal{T}}(u,\beta)\|_{X^s\times\R^{n+1}}
\le C\,\Big( \|u\|_{X^s}+\|\beta\|_{\R^{n+1}}\Big).
\end{equation}
To prove it, we first use~\eqref{ass0} to bound~$
\|\tilde{\mathcal{T}}(u,\beta)\|_{\dot{H}^s(\mathbb{R}^n)\times\R^{n+1}}$
with~$[u]_{\dot{H}^s(\mathbb{R}^n)} +\|\beta\|_{\R^{n+1}}$,
and then we observe that the latter quantity is in turn bounded
by~$\|u\|_{X^s}+\|\beta\|_{\R^{n+1}}$. Thus,
in order to show that~${\mathcal{T}}$ is bounded as an operator over~$X^s\times\R^{n+1}$,
we only have to bound~$
\|\tilde{\mathcal{T}}(u,\beta)\|_{L^\infty(\R^n)\times\R^{n+1}}$.

That is to say that the desired result is proved if we show that,
for any~$u\in X^s$ and any~$\beta\in\R^{n+1}$ we have that
\begin{equation}\label{goal}
\|\tilde{\mathcal{T}}(u,\beta)\|_{L^\infty(\R^n)\times\R^{n+1}}
\le C\,\Big( \|u\|_{X^s}+\|\beta\|_{\R^{n+1}}\Big).
\end{equation}
To prove this, we fix~$u\in X^s$ and~$\beta\in\R^{n+1}$
and we set~$(v,\alpha):=\tilde{\mathcal{T}}(u,\beta)\in\dot{H}^s(\R^n)\times\R^{n+1}$.
Thus, by~\eqref{OP},
\begin{equation}\label{W3DGRa}
X^s\times\R^{n+1}\ni (u,\beta)={\mathcal{T}}(v,\alpha)=\left(
Tv-\sum_{i=1}^{n+1}\alpha_i q_i, \langle v,q_1\rangle,\dots,
\langle v,q_{n+1}\rangle\right).\end{equation}
Taking the first coordinate and using~\eqref{spi}, we obtain that,
for any~$j\in\{1,\dots,n+1\}$,
$$ \langle u,q_j\rangle
=\langle Tv-\sum_{i=1}^{n+1}\alpha_i q_i,\,q_j\rangle=
-\sum_{i=1}^{n+1}\alpha_i \langle q_i,\,q_j\rangle.$$
Thus, by Lemma~\ref{eigen}, we have
that~$\langle u,q_j\rangle=-\alpha_j \lambda_j$ and therefore
$$ |\alpha_j|\le C \,[u]_{\dot H^s(\R^n)}.$$
Accordingly
\begin{equation}\label{SERVIRA}
\|\alpha\|_{\R^{n+1}}\le C\,\|u\|_{X^s}.
\end{equation}
Now we set~$\psi:=v-u$. Notice that~$\psi\in\dot{H}^s(\R^n)$, since so are~$u$ and~$v$.
Moreover, taking the first coordinate in~\eqref{W3DGRa}
and using~\eqref{DE T}
and~\eqref{EL-4}, we see that~$\psi$ is a weak solution of
\begin{eqnarray*}
(-\Delta)^s \psi &=& (-\Delta)^s v-(-\Delta)^s u\\
&=& (-\Delta)^s v
-(-\Delta)^s Tv+\sum_{i=1}^{n+1}\alpha_i (-\Delta)^s q_i\\
&=& (-\Delta)^s J(A_0'(z_{\mu,\xi})v)+\sum_{i=1}^{n+1}\alpha_i (-\Delta)^s q_i\\
&=& pz_{\mu,\xi}^{p-1}v+p\sum_{i=1}^{n+1} \alpha_i z_{\mu,\xi}^{p-1}q_i\\
&=& pz_{\mu,\xi}^{p-1}\psi+pz_{\mu,\xi}^{p-1}u
+p\sum_{i=1}^{n+1} \alpha_i z_{\mu,\xi}^{p-1}q_i.
\end{eqnarray*}
The reader may check that this agrees with~\eqref{7sdvfbrt5}.
Furthermore, by~\eqref{ass0},
\begin{eqnarray*}
[v]_{\dot{H}^s(\mathbb{R}^n)} &\le&
\|(v,\alpha)\|_{\dot{H}^s(\mathbb{R}^n)\times\R^{n+1}}
\\ &=&
\|\tilde{\mathcal{T}}(u,\beta)\|_{\dot{H}^s(\mathbb{R}^n)\times\R^{n+1}}
\\ &\le&
C\,\Big( [u]_{\dot{H}^s(\mathbb{R}^n)}+
\|\beta\|_{\R^{n+1}}\Big).
\end{eqnarray*}
Consequently,
$$ [\psi]_{\dot{H}^s(\mathbb{R}^n)} \le
[u]_{\dot{H}^s(\mathbb{R}^n)} +[v]_{\dot{H}^s(\mathbb{R}^n)}
\le C\,\Big( [u]_{\dot{H}^s(\mathbb{R}^n)}+
\|\beta\|_{\R^{n+1}}\Big),$$
up to renaming constants.
The reader may check that this implies~\eqref{qwrhgsa44}.
Accordingly the assumptions of Lemma~\ref{qwrhgsa44LEMMA}
are satisfied, and we deduce from it that
$$ \|\psi\|_{L^\infty(\R^n)}\le C\Big(
\|u\|_{X^s}+\|\alpha\|_{\R^{n+1}}+\|\beta\|_{\R^{n+1}}\Big).$$
Consequently, using~\eqref{SERVIRA}, we obtain that
\begin{eqnarray*}
\|v\|_{L^\infty(\R^n)} &\le& \|u\|_{L^\infty(\R^n)}
+ \|\psi\|_{L^\infty(\R^n)}
\\ &\le& C\,\Big(
\|u\|_{X^s}+\|\alpha\|_{\R^{n+1}}+\|\beta\|_{\R^{n+1}}\Big)\\
&\le& C\,\Big(
\|u\|_{X^s}+\|\beta\|_{\R^{n+1}}\Big),
\end{eqnarray*}
up to renaming constants. Using this and once again~\eqref{SERVIRA},
we obtain that
\begin{eqnarray*} &&\|\tilde{\mathcal{T}}(u,\beta)\|_{L^\infty(\R^n)\times\R^{n+1}}
=\|(v,\alpha)\|_{L^\infty(\R^n)\times\R^{n+1}}\\&&\qquad=
\|v\|_{L^\infty(\R^n)}+\|\alpha\|_{\R^{n+1}}
\le C\,\Big( \|u\|_{X^s}+\|\beta\|_{\R^{n+1}}\Big).\end{eqnarray*}
This establishes~\eqref{goal} and in turn~\eqref{goal1},
and so it completes the proof of the desired result.
\end{proof}

\subsubsection{Proof of Lemma \ref{lemma0}}

Once we have studied in detail the operator $H$, we can prove Lemma \ref{lemma0}.
As we pointed out at the beginning of this subsection,
the idea is to do it by means of the Implicit Function Theorem.
For the sake of completeness, we write here the precise statement
of this theorem that we will use (see Theorem 2.3, page 38, of \cite{AP}).

\begin{theorem}[Implicit Function Theorem]\label{IFT}
Let $X, Y, Z$ be Banach spaces, and let~$\Lambda$ and~$U$ be
open sets of~$X$ and~$Y$ respectively. Let~$H\in C^1(\Lambda\times U,Z)$ and
suppose that~$H(\lambda^*,u^*)=0$ and~$\displaystyle \frac{\partial H}{\partial u}(\lambda^*,u^*)\in Inv(Y, Z)$.

Then there exist neighborhoods~$\Theta$ of~$\lambda^*$ in~$X$
and~$U^*$ of~$u^*$ in~$Y$, and a map~$g\in C^1(\Theta, Y)$ such that
\begin{itemize}
\item[a)] $H(\lambda, g(\lambda))=0$, for all $\lambda\in \Theta$.
\item[b)] $H(\lambda, u)=0$, with $(\lambda,u)\in \Theta\times U^*$, implies $u=g(\lambda)$.
\item[c)] $\displaystyle g'(\lambda)=-\left(\frac{\partial H}{\partial u}(p)\right)^{-1}\circ\frac{\partial H}{\partial\lambda}(p)$, where $p=(\lambda, g(\lambda))$ and $\lambda\in \Theta$.
\end{itemize}

\end{theorem}

Now we conclude the proof of Lemma \ref{lemma0}.

\begin{proof}[Proof of Lemma \ref{lemma0}]
Consider $H$ defined in \eqref{defH}. First we observe that $H$ is $C^1$ with respect to $\mu$ and $\xi$.
Indeed, $z_{\mu,\xi}$ is $C^1$ with respect to $\mu$ and $\xi$.
Moreover, $J$ is linear and $A_\epsilon(z_{\mu,\xi}+w)$ is $C^1$
with respect to $z_{\mu,\xi}$ since $z_{\mu,\xi}+w$ is
bounded from zero
on the support of $h$
(recall \eqref{PP8}), therefore $H_1$ is $C^1$ with respect to $z_{\mu,\xi}$.

Also, $H$ is $C^1$ with respect to $\epsilon$ and $\alpha$,
since it depends linearly on these variables (recall that $J$ is linear
and $A_\epsilon$ is linear with respect to $\epsilon$).
Finally, $H$ is~$C^1$ with respect to $w$ thanks to Lemma \ref{C1w}.

Now we use the Implicit Function Theorem.
Indeed, we notice that
\begin{equation}\label{H1zero}
 H_1(\mu,\xi,0,0,0)=z_{\mu,\xi}-J(A_0(z_{\mu,\xi}))=z_{\mu,\xi}-J(z_{\mu,\xi}^p)=0,
\end{equation}
since $z_{\mu,\xi}$ is a solution to \eqref{entire} (recall also \eqref{EL-4}). Moreover,
\begin{equation}\label{H2zero}
 H_2(\mu,\xi,0,0,0)=0.
\end{equation}
In order to follow the notation of Theorem \ref{IFT}, we set
$$X:=\mathbb{R}\times \mathbb{R}^n\times \mathbb{R},\quad Y:=X^s\times \mathbb{R}^{n+1},\quad Z:= X^s\times\mathbb{R}^{n+1},$$
$$\Lambda:=(\mu_1,\mu_2)\times B_R\times\mathbb{R}, \quad U:=V\times\mathbb{R}^{n+1},$$
and
$$\lambda^*:=(\mu,\xi,0),\quad u^*:=(0,0),\quad u:=(w,\alpha).$$
Thus, we have proved that
\begin{itemize}
\item[(i)] $H\in C^1(\Lambda\times U,Z)$, by the linear dependance of the variables and Lemma~\ref{C1w};
\item[(ii)] $H(\lambda^*,u^*)=0$, by~\eqref{H1zero} and \eqref{H2zero};
\item[(iii)] $\displaystyle \frac{\partial H}{\partial u}(\lambda^*,u^*)\in Inv(Y, Z)$,
by \eqref{OP}, \eqref{OP secondo} and Proposition~\ref{invertibility}.
\end{itemize}
\noindent Notice here that, since $V$ was defined as
$$V:= \{w\in X^s {\mbox{ s.t }} \|w\|_{X^s}<a/2\}, $$
it is an open subset of $X^s$.
Therefore,
all the hypotheses of the Implicit Function Theorem are satisfied, and we conclude the existence of a nontrivial
$w\in X^s$ solution to \eqref{f6v7gb8idsfvvcwqw}, that is, there exists $w\in X^s\cap \left(T_{z_{\mu,\xi}}Z_0\right)^{\perp}$
that solves the auxiliary equation in \eqref{AUX}. Furthermore, since $H$ is of class
$C^1$ with respect to $\epsilon$, $\mu$ and $\xi$ in $X^s$,
we deduce that so is $w$.

Now we focus on the proof of~\eqref{est}. We observe that
\begin{equation}\label{der w eps alpha}
\left\| \frac{\partial (w,\alpha)}{\partial\eps}\right\|_{X^s \times \R^{n+1}} \le C.
\end{equation}
Indeed, we write
\begin{equation}\label{67uiddasdf}
H \big(\mu,\,\xi,\, w(\eps,z_{\mu,\xi}),\,\eps,\,\alpha(\eps,z_{\mu,\xi})\big)=0,
\end{equation}
we differentiate with respect to $\eps$ and we set $\eps:=0$.
Notice tha
we are using the order of the variables of $H$ as given in \eqref{defH}.

Since
\begin{equation}\label{6yufjgyuikxmmncmw}
{\mbox{$w(0,z_{\mu,\xi})=0$ and $\alpha(0,z_{\mu,\xi})=0$,}}\end{equation}
we obtain that
$$ \frac{\partial H}{\partial \eps} (\mu,\xi,0,0,0)
+\frac{\partial H}{\partial (w,\alpha)}(\mu,\xi,0,0,0)
\frac{\partial (w,\alpha)}{\partial \eps} (0,z_{\mu,\xi})=0.$$
Therefore, using the invertibility assumption,
we get that
$$ \frac{\partial (w,\alpha)}{\partial \eps} (0,z_{\mu,\xi})
=-
\left(\frac{\partial H}{\partial (w,\alpha)}(\mu,\xi,0,0,0)\right)^{-1}
\frac{\partial H}{\partial \eps} (\mu,\xi,0,0,0),$$
and so, since $H$ is $C^1$ with respect to $X^s$,
$$
\left\|\frac{\partial (w,\alpha)}{\partial \eps} 
(0,z_{\mu,\xi})\right\|_{X^s \times \R^{n+1}}\le C.$$
Then, since $(w,\alpha)$ is $C^1$ in $\eps$, in virtue of the Implicit Function
Theorem, we obtain 
$$ \left\|\frac{\partial (w,\alpha)}{\partial \eps} 
(\eps,z_{\mu,\xi})\right\|_{X^s \times \R^{n+1}}\le C+
\left\|\frac{\partial (w,\alpha)}{\partial \eps} 
(0,z_{\mu,\xi})\right\|_{X^s \times \R^{n+1}}\le C,$$
up to renaming $C>0$, and this gives
\eqref{der w eps alpha}.

{F}rom \eqref{der w eps alpha} and \eqref{6yufjgyuikxmmncmw} we obtain that
\begin{equation*}
\| (w,\alpha)\|_{X^s \times \R^{n+1}}\le C\eps,
\end{equation*}
and this implies the first estimate in \eqref{est}.

Now we prove the second and third estimates in \eqref{est}.
In this case, we will see that the roles of $\mu$ and $\xi$ are basically the same:
for this, we write $\varpi\in\R$ for any of the variables $(\mu,\xi)\in\R^{n+1}$
and we use the linearized equation to see that
$$ (-\Delta)^s \frac{\partial z_{\mu,\xi}}{\partial\varpi}=p z_{\mu,\xi}^{p-1}
\frac{\partial z_{\mu,\xi}}{\partial\varpi}.$$
This information can be written as
\begin{equation*}
\frac{\partial H}{\partial\varpi} (\mu,\xi,0,0,0)=0.
\end{equation*}
Now we take derivatives of \eqref{67uiddasdf}
with respect to $\varpi$ and we set $\eps:=0$.
Recalling \eqref{6yufjgyuikxmmncmw} we obtain that
\begin{eqnarray*}
0 &=& \frac{\partial H}{\partial \varpi}(\mu,\xi,0,0,0)
+\frac{\partial H}{\partial (w,\alpha)}(\mu,\xi,0,0,0)
\frac{\partial (w,\alpha)}{\partial \varpi}(0,z_{\mu,\xi})\\
&=&\frac{\partial H}{\partial (w,\alpha)}(\mu,\xi,0,0,0)
\frac{\partial (w,\alpha)}{\partial \varpi}(0,z_{\mu,\xi}).
\end{eqnarray*}
Hence, from the invertibility condition, we conclude that
$$ \frac{\partial (w,\alpha)}{\partial \varpi}(0,z_{\mu,\xi})=0.$$
Since $(w,\alpha)$ are $C^1$ in $\eps$, we obtain that
$$ \lim_{\eps\to0}
\left\|\frac{\partial (w,\alpha)}{\partial \varpi}(\eps,z_{\mu,\xi})\right\|_{X^s \times \R^{n+1}}=0.$$
This gives the second and third claim in \eqref{est}
and completes the proof of Lemma \ref{lemma0}.
\end{proof}

\subsection{Finite-dimensional reduction}
Up to this point, we have found a function $w$ so that $z_{\mu,\xi}+w$
satisfies our problem in the weak sense,
when we test with functions $\varphi\in (T_{z_{\mu,\xi}}Z_0)^\perp\cap X^s$.
The following result states that actually the equation is satisfied
for every test function in $X^s$,
i.e. that $z_{\mu,\xi}+w$ is a solution to \eqref{problem}.

Indeed, consider the reduced functional $\Phi_\varepsilon:Z_0\rightarrow \mathbb{R}$, defined by
\begin{equation*}
\Phi_\varepsilon(z):=f_\varepsilon(z+w),
\end{equation*}
where $w=w(\varepsilon,z)$ is provided by Lemma \ref{lemma0}.
\begin{prop}
Suppose that $\Phi_\varepsilon$ has a critical point $z_{\mu^\varepsilon,\xi^\varepsilon}\in Z_0$ for $\varepsilon$ small enough. Thus, $z_{\mu^\varepsilon,\xi^\varepsilon}+w$ is a critical point of $f_\varepsilon$, where $w=w(\varepsilon,z_{\mu^\varepsilon,\varepsilon^\varepsilon})\in (T_{z_{\mu^\varepsilon,\xi^\varepsilon}}Z_0)^\perp$ is provided by Lemma \ref{lemma0}.
\end{prop}

\begin{proof}
For simplicity, we will denote $\mu:=\mu^{\varepsilon}$ and $\xi:=\xi^{\varepsilon}$,
and thus $z_{\mu,\xi}:=z_{\mu^{\varepsilon},\xi^{\varepsilon}}$.
Since $z_{\mu,\xi}$ is a critical point of $\Phi_\varepsilon$,
we know that there exists $\varepsilon_0>0$
such that for every $0<\varepsilon<\varepsilon_0$ and every $\varphi\in (T_{z_{\mu,\xi}}Z_0)\cap X^s$ it holds
\begin{equation}\label{PartialPhi}
\displaystyle\frac{d}{dt}\Phi_\varepsilon(\psi(t))\bigg|_{t=0}=0,
\end{equation}
where $\psi:[0,1]\rightarrow Z_0$ is a curve in $Z_0$ such that $\psi(0)=z_{\mu,\xi}$ and $\psi'(0)=\varphi$. Recalling the definition of $\Phi_\varepsilon$, we observe that
\begin{equation*}\begin{split}
\displaystyle \frac{d}{dt}\Phi_\varepsilon&(\psi(t))\bigg|_{t=0}
=\frac{d}{dt}f_\varepsilon(\psi(t)+w(\varepsilon, \psi(t)))\bigg|_{t=0}
\\ =&\,\frac{d}{dt}\left[f_\varepsilon(\psi(0)+w(\varepsilon, \psi(0))+t\cdot f_\varepsilon'(\psi(0)+w(\varepsilon, \psi(0)))\left(\psi'(0)+\frac{\partial w}{\partial z_{\mu,\xi}}\psi'(0)\right)\right]_{t=0}
\\ =&\,f_\varepsilon'(z_{\mu,\xi}+w(\varepsilon,z_{\mu,\xi}))\left(\varphi+\frac{\partial w}{\partial z_{\mu,\xi}}\varphi\right),
\end{split}\end{equation*}
and hence \eqref{PartialPhi} is equivalent to
\begin{equation}\label{criticalPointPhi}\begin{split}
&\displaystyle\iint_{\R^{2n}}{\frac{\big( (z_{\mu,\xi}+w)(x)-(z_{\mu,\xi}+w)(y)\big)\,\big(
(\varphi+\frac{\partial w}{\partial z_{\mu,\xi}}\varphi)(x)-(\varphi+\frac{\partial w}{\partial z_{\mu,\xi}} \varphi)(y)\big)}{|x-y|^{n+2s}}\,dx\,dy}
\\ &\qquad\,=\, \int_{\R^n} \Big( \epsilon h(x) \big( z_{\mu,\xi}(x)+w(x)\big)^q
+\big( z_{\mu,\xi}(x)+w(x)\big)^p \Big)
\,(\varphi+\frac{\partial w}{\partial z_{\mu,\xi}}\varphi)(x)\,dx,
\end{split}\end{equation}
for any $\varphi\in \left(T_{z_{\mu,\xi}}Z_0\right)\cap X^s$.

Moreover, since $w$ solves \eqref{f6v7gb8idsfvvcwqw}, $H_1(\mu,\xi,w,\varepsilon,\alpha)=0$ is equivalent to affirm that
\begin{equation}\label{fPrimeTangent}\begin{split}
&\iint_{\R^{2n}} \frac{\big( (z_{\mu,\xi}+w)(x)-(z_{\mu,\xi}+w)(y)\big)\,\big(
\phi(x)-\phi(y)\big)}{|x-y|^{n+2s}}\,dx\,dy
\\ &\qquad\,-\, \int_{\R^n} \Big( \epsilon h(x) \big( z_{\mu,\xi}(x)+w(x)\big)^q
+\big( z_{\mu,\xi}(x)+w(x)\big)^p \Big)
\,\phi(x)\,dx,
\\ &\qquad\,=\, \sum_{i=1}^{n+1} {\alpha_i}\iint_{\R^{2n}} \frac{\big( q_i(x)-q_i(y)\big)\,\big(
\phi(x)-\phi(y)\big)}{|x-y|^{n+2s}}\,dx\,dy,
\end{split}\end{equation}
for any $\phi\in  X^s$.

Consider now $q_j\in T_{z_{\mu,\xi}}Z_0$ defined in \eqref{qj}.
Thus, taking $\varphi:={q_j}$ in \eqref{criticalPointPhi} and applying
\eqref{fPrimeTangent} with 
$\phi:=q_j+\frac{\partial w}{\partial z_{\mu,\xi}}q_j\in X^s$, we obtain
\begin{equation}\label{linearSystem}\begin{split}
&\displaystyle 0 \,=\, \sum_{i=1}^{n+1} {\alpha_i}\iint_{\R^{2n}} \frac{\big( q_i(x)-q_i(y)\big)\,\big(
({q_j}+\frac{\partial w}{\partial z_{\mu,\xi}}{q_j})(x)-({q_j}+\frac{\partial w}{\partial z_{\mu,\xi}} {q_j})(y)\big)}{|x-y|^{n+2s}}\,dx\,dy\\
&\qquad \,=\, \sum_{i=1}^{n+1} {\alpha_i}\langle q_i,{q_j}\rangle + \sum_{i=1}^{n+1}{\alpha_i}\langle q_i,\frac{\partial w}{\partial z_{\mu,\xi}} {q_j}\rangle\\
&\qquad \,=\, {\lambda_j \alpha_j}+\sum_{i=1}^{n+1}
{\alpha_i}\langle q_i, \frac{\partial w}{\partial z_{\mu,\xi}}{q_j}\rangle,
\end{split}\end{equation}
where Lemma~\ref{eigen} was also used in the last line.

Set now the $(n+1)\times(n+1)$ matrix $B^\varepsilon=(b_{ij}^\varepsilon)$, defined as
\begin{eqnarray*}
b_{ij}^\varepsilon &:=& \langle q_i,\frac{\partial w}{\partial\xi_j}\rangle,\quad i=1,\ldots,n+1,\,\,j=1,\ldots,n,\\
b_{i,n+1}^\varepsilon &:=&\langle q_i,\frac{\partial w}{\partial\mu}\rangle,\quad i=1,\ldots,n+1.
\end{eqnarray*}
By Cauchy-Schwartz inequality and \eqref{est} one has
\begin{equation}\label{DerivLimit}
\lim_{\varepsilon\rightarrow 0}{\langle q_i, \frac{\partial w}{\partial\xi_j}\rangle}=\lim_{\varepsilon\rightarrow 0}{\langle q_i,\frac{\partial w}{\partial\mu}\rangle}=0,\quad i=1,\ldots,n+1,\,\,j=1,\ldots,n,
\end{equation}
and thus $\displaystyle\lim_{\varepsilon\rightarrow 0}\|B^\varepsilon\|=0$. Recalling that
$$\displaystyle\frac{\partial w}{\partial z_{\mu,\xi}} q_j=\frac{\partial w}{\partial z_{\mu,\xi}}\frac{\partial z_{\mu,\xi}}{\partial\xi_j}=\frac{\partial}{\partial \xi_j}w(\varepsilon,z_{\mu,\xi})=\frac{\partial w}{\partial\xi_j}\hbox{ for }j=1,\ldots, n$$
and
$$\displaystyle\frac{\partial w}{\partial z_{\mu,\xi}} q_{n+1}=\frac{\partial w}{\partial z_{\mu,\xi}} \frac{\partial z_{\mu,\xi}}{\partial\mu}=\frac{\partial}{\partial \mu}w(\varepsilon,z_{\mu,\xi})=\frac{\partial w}{\partial\mu},$$
equation \eqref{linearSystem} becomes
$$\lambda_j \alpha_j+\sum_{i=1}^{n+1}
\alpha_i b_{ij}^\varepsilon =0,\quad i,j=1,\ldots,n+1,$$
that is nothing but a $(n+1)\times(n+1)$ linear system
with associated matrix $\lambda\,Id_{\mathbb{R}^{n+1}}+B^\varepsilon$,
whose entries are $\lambda_j\delta_{ij}+b_{ij}^\varepsilon$,
where $\delta_{jj}=1$ and $\delta_{ij}=0$ whether $i\neq j$.
Thus, since $\displaystyle\lim_{\varepsilon\rightarrow 0}\|B^\varepsilon\|=0$,
there exists $\varepsilon_1>0$ such that for $\varepsilon<\varepsilon_1$
the matrix $\lambda\,Id_{\mathbb{R}^{n+1}}+B^\varepsilon$ is invertible, and therefore $\alpha_i=0$ for every $i=1,\ldots,n+1$.
Hence, coming back to \eqref{fPrimeTangent}, we get
\begin{equation*}\begin{split}
&\iint_{\R^{2n}} \frac{\big( (z_{\mu,\xi}+w)(x)-(z_{\mu,\xi}+w)(y)\big)\,\big(
\phi(x)-\phi(y)\big)}{|x-y|^{n+2s}}\,dx\,dy
\\ &\qquad\,=\, \int_{\R^n} \Big( \epsilon h(x) \big( z_{\mu,\xi}(x)+w(x)\big)^q
+\big( z_{\mu,\xi}(x)+w(x)\big)^p \Big)
\,\phi(x)\,dx,
\end{split}\end{equation*}
for every $\phi\in  X^s$, that is, $z_{\mu,\xi}+w$ is a critical point of $f_\varepsilon$.
\end{proof}

\section{Study of the behavior of $\Gamma$}\label{sec:gamma}

At this point, we have reduced our original problem to a finite-dimensional one.
Indeed, we define the perturbed manifold
$$ Z_\epsilon:= \{u:=z_{\mu,\xi}+w(\epsilon, z_{\mu,\xi}) {\mbox{ s.t. }} z_{\mu,\xi}\in Z_0\},$$
which is a natural constraint for the functional $f_\epsilon$.

We recall that $G$ and $U$
have been defined in~\eqref{Gu}
and~\eqref{defU}, respectively, and we give the following
\begin{defn}
We say that $u\in U$ is a proper local maximum (or minimum, respectively) of $G$ if
there exists a neighborhood $\mathcal{U}$ of $u$ such that
$$ G(u)\geq G(v)\;\;\forall\,v\in\mathcal{U}\qquad \quad(G(u)\leq G(v)\;\;\forall\,v\in\mathcal{U},
{\mbox{ respectively}}),$$
and
$$ G(u)>\sup_{v\in\partial\mathcal{U}}G(v) \qquad\quad (G(u)<\inf_{v\in\partial\mathcal{U}}G(v),
{\mbox{ respectively}}).$$
\end{defn}

With this, one can prove that:
\begin{prop}\label{prop1}
Suppose that $z_{\mu,\xi}\in Z_0$ is a proper local maximum or minimum of $G$.
Then, for $\epsilon>0$ sufficiently small,
$u_\epsilon:=z_{\mu,\xi}+w(\epsilon, z_{\mu,\xi})\in Z_\epsilon$ is a critical point
of $f_\epsilon$.
\end{prop}

The proof of this can be found for instance in \cite{AM}
(see in particular Theorem 2.16 there).
A simple explanation goes as follows.
First we notice that, for any $z_{\mu,\xi}\in Z_0$,
\begin{equation}\label{zero}
f'_0(z_{\mu,\xi})=0,
\end{equation}
where $f_0$ is defined in \eqref{fzero}. Indeed, $z_{\mu,\xi}$ is a critical point
of $f_0$, being a solution to \eqref{entire}.
Now, recalling \eqref{perturbed} and
using Taylor expansion in the vicinity of $z_{\mu,\xi}$, we have
\begin{eqnarray*}
f_\epsilon(z_{\mu,\xi}+ w) &=& f_0(z_{\mu,\xi}+w)-\epsilon\,G(z_{\mu,\xi}+w) \\
&=& f_0(z_{\mu,\xi})+ f'_0(z_{\mu,\xi})\,w + o(\|w\|_{X^s}) -\epsilon\,G(z_{\mu,\xi})
-\epsilon\, G'(z_{\mu,\xi})\,w + o(\epsilon)\\
&=& f_0(z_{\mu,\xi}) -\epsilon\, G(z_{\mu,\xi}) +o(\epsilon)\\
&=& f_0(z_0) -\epsilon\, G(z_{\mu,\xi}) +o(\epsilon),
\end{eqnarray*}
where we have used \eqref{zero} and \eqref{est}, and the translation and dilation invariance of $f_0$.

Therefore, we have reduced our problem to find critical points of $G$.
For this, we set
\begin{equation}\label{gamma}
\Gamma(\mu,\xi):= G(z_{\mu,\xi})= \frac{\mu^{-\gamma_s}}{q+1}\,
\int_{\R^n}h(x) z_0^{q+1}\left(\frac{x-\xi}{\mu}\right)\,dx,
\end{equation}
where
\begin{equation}\label{r00}
\gamma_s:=\frac{(n-2s)(q+1)}{2}.
\end{equation}

Now we prove some lemma concerning the behavior of $\Gamma$.
In the first one we compute the limit of $\Gamma$ as $\mu$ tends to zero.

\begin{lemma}\label{lemma1}
Let $\Gamma$ be as in \eqref{gamma}. Then
$$ \lim_{\mu\rightarrow 0}\Gamma(\mu,\xi)=0 \ {\mbox{ uniformly in }}\xi. $$
\end{lemma}

\begin{proof} Thanks to \eqref{h2}, there exists $r>1$ such that
\begin{equation}\label{supp}
\omega=supp\, h\subset B_r.
\end{equation}
We first suppose that $\xi\in\R^n$ is such that $|\xi|\ge 2r$.
Therefore, if $|y|<r$ then
$$ |\xi +y|\ge |\xi|-|y|>r,$$
and so $\xi +y\in B_r^c\subset\omega^c$. This implies that
\begin{equation}\label{r60}
h(y+\xi)=0 \ {\mbox{ if }} |\xi|\ge 2r  {\mbox{ and }} |y|<r.
\end{equation}
Now, we observe that, using the change of variable $y=x-\xi$, $\Gamma$ can be written as
$$ \Gamma(\mu,\xi)=\frac{\mu^{-\gamma_s}}{q+1}\,\int_{\R^n}h(y+\xi)\,
z_0^{q+1}\left(\frac{y}{\mu}\right)\,dy.$$
Hence, using \eqref{r60} we have that, if $|\xi|\ge 2r$,
\begin{eqnarray*}
\Gamma(\mu,\xi) &=& \frac{\mu^{-\gamma_s}}{q+1}\, \int_{|y|\ge r}h(y+\xi)\,
z_0^{q+1}\left(\frac{y}{\mu}\right)\,dy \\
&\le & \frac{\mu^{-\gamma_s}}{
q+1}\,\max_{|y|\ge r}z_0^{q+1}\left(\frac{y}{\mu}\right)\,
\int_{|y|\ge r}\big|h(y+\xi)\big|\, dy.
\end{eqnarray*}
This implies that
\begin{equation}\label{r61}
|\Gamma(\mu,\xi)|\le \frac{\mu^{-\gamma_s}}{q+1}\,
\max_{|y|\ge r}z_0^{q+1}\left(\frac{y}{\mu}\right)\, \|h\|_{L^1(\R^n)}.
\end{equation}
Now, recalling \eqref{zetazero}, we obtain that
$$ z_0^{q+1}\left(\frac{y}{\mu}\right)=
\alpha_{n,s}^{q+1}\frac{\mu^{(n-2s)(q+1)}}{(\mu^2+|y|^2)^{\frac{(n-2s)(q+1)}{2}}},
$$
and so
$$ \max_{|y|\ge r}z_0^{q+1}\left(\frac{y}{\mu}\right)= \mu^{(n-2s)(q+1)}\,
\max_{|y|\ge r}\frac{\alpha_{n,s}^{q+1}}{(\mu^2+|y|^2)^{\frac{(n-2s)(q+1)}{2}}} \le C\, \mu^{(n-2s)(q+1)},$$
for a suitable constant $C>0$ independent on $\mu$.
Using this in \eqref{r61} and recalling \eqref{r00}, \eqref{h2} and
the fact that $h$ is continuous, we get
(up to renaming $C$)
$$ |\Gamma(\mu,\xi)|\le C \, \mu^{\frac{(n-2s)(q+1)}{2}},$$
which tends to zero as $\mu\rightarrow 0$. This concludes the proof in the case $|\xi|\ge 2r$.

If instead $|\xi|<2r$ then one has
\begin{equation}\begin{split}\label{s00}
\int_{\R^n}h(x)\,z_0^{q+1}\left(\frac{x-\xi}{\mu}\right)\,dx &= \,
\int_{|x|<r}h(x)\,z_0^{q+1}\left(\frac{x-\xi}{\mu}\right)\,dx \\
&\le \, \|h\|_{L^\infty(\R^n)}\, \int_{|x|<r}z_0^{q+1}\left(\frac{x-\xi}{\mu}\right)\,dx,
\end{split}\end{equation}
thanks to \eqref{supp}, \eqref{h2} and the fact that $h$ is continuous.

We claim that
\begin{equation}\label{est mu}
\int_{|x|<r}z_0^{q+1}\left(\frac{x-\xi}{\mu}\right)\,dx \le C\, \mu^{\min\{n,(n-2s)(q+1)\}},
\end{equation}
for some positive constant $C$ independent of $\mu$ (possibly depending on $r$).
To prove this, we recall \eqref{zetazero} and we get
\begin{eqnarray*}&&
\int_{|x|<r}z_0^{q+1}\left(\frac{x-\xi}{\mu}\right)\,dx \\&=& \alpha_{n,s}^{q+1}
\,\int_{|x|<r}\frac{dx}{\left(1+\frac{|x-\xi|^2}{\mu^2}\right)^{\frac{(n-2s)(q+1)}{2}}} \\
&\le & \alpha_{n,s}^{q+1}\,\left(\int_{|x-\xi|\le\mu}\,dx +
\int_{\mu<|x-\xi|<3r}\frac{\mu^{(n-2s)(q+1)}}{|x-\xi|^{(n-2s)(q+1)}}\,dx\right) \\
&\le & C\,\left(\mu^n +\mu^{(n-2s)(q+1)}\, \int_\mu^{3r}\rho^{n-1-(n-2s)(q+1)}\,d\rho \right) \\
&\le & C\,\left(\mu^n + \mu^{(n-2s)(q+1)}\,\mu^{-[(n-2s)(q+1)-n]_+}\right) \\
&\le & C\,\left(\mu^n + \mu^{\min\{n,(n-2s)(q+1)\}}\right) \\
&\le & C\,\mu^{\min\{n,(n-2s)(q+1)\}},
\end{eqnarray*}
up to changing $C$ from line to line, and this shows \eqref{est mu}.
Therefore, by \eqref{gamma}, \eqref{r00} and \eqref{s00} we have that
$$ |\Gamma(\mu,\xi)| \le C\, \mu^{-\frac{(n-2s)(q+1)}{2}}\,\mu^{\min\{n,(n-2s)(q+1)\}}.$$
Hence, if $(n-2s)(q+1)\le n$ we get that
$$ |\Gamma(\mu,\xi)| \le C\, \mu^{(n-2s)(q+1)}, $$
which implies that $\Gamma(\mu,\xi)$ tends to zero as $\mu\to0$.
If instead $n<(n-2s)(q+1)$ we obtain that
$$ |\Gamma(\mu,\xi)| \le C\, \mu^{n-\frac{(n-2s)(q+1)}{2}}.$$
In this case, we observe that, since $q\in(0,p)$ with $p=\frac{n+2s}{n-2s}$, then
$q+1<\frac{2n}{n-2s}$, and so
$$ n-\frac{(n-2s)(q+1)}{2}>n-\frac{n-2s}{2}\frac{2n}{n-2s}=0.$$
This implies that also in this case $\Gamma(\mu,\xi)$ tends to zero as $\mu\to0$.
This concludes the proof of Lemma \ref{lemma1}.
\end{proof}

Now we compute the limit of $\Gamma$ as $\mu+|\xi|$ tends to $+\infty$.

\begin{lemma}\label{lemma2}
Let $\Gamma$ be as in \eqref{gamma}. Then
$$ \lim_{\mu+|\xi|\rightarrow +\infty}\Gamma(\mu,\xi)=0.$$
\end{lemma}

\begin{proof}
Suppose that $\mu\rightarrow +\infty$. Then recalling \eqref{h2},
the fact that $h$ is continuous and \eqref{zetazero} we have
$$ |\Gamma(\mu,\xi)|\le C\,\mu^{-\gamma_s}\,
\|h\|_{L^1(\R^n)}, $$
for some positive constant $C$ independent on $\mu$.
Therefore $\Gamma(\mu,\xi)$ tends to zero as $\mu\rightarrow +\infty$.

Now suppose that $\mu\rightarrow\bar{\mu}$ for some $\bar{\mu}\in[0,+\infty)$,
therefore $|\xi|\rightarrow +\infty$. If $\bar{\mu}=0$, then we can use
Lemma \ref{lemma1} and we get the desired result. Hence,
we can suppose that $\bar{\mu}\in(0,+\infty)$.
In this case, we make the change of variable $y=x-\xi$ and we write $\Gamma$ as
\begin{equation}\label{hgfhh}
\Gamma(\mu,\xi)= \frac{\mu^{-\gamma_s}}{q+1}\,\int_{\R^n}h(y+\xi)\,
z_0^{q+1}\left(\frac{y}{\mu}\right)\,dy.
\end{equation}
Since $h$ has compact support (recall \eqref{h2}), there exists $r>0$ such that
$\omega=supp\,h\subset B_r$ and so \eqref{hgfhh} becomes
\begin{equation}\label{conto2}
\Gamma(\mu,\xi)= \frac{\mu^{-\gamma_s}}{q+1}\,\int_{|y+\xi|\le r}h(y+\xi)\,
z_0^{q+1}\left(\frac{y}{\mu}\right)\,dy.
\end{equation}
We also notice that, since $|\xi|\rightarrow +\infty$, we can suppose that $|\xi|>2r$.
Therefore, if $y\in B_r(-\xi)$, then $|y+\xi|\le r<|\xi|/2$, which implies that
$$ |y|\ge |\xi|-|y+\xi|\ge |\xi|-\frac{|\xi|}{2}=\frac{|\xi|}{2}.$$
Hence, recalling \eqref{zetazero}, we obtain that if $y\in B_r(-\xi)$
\begin{eqnarray*}
z_0^{q+1}\left(\frac{y}{\mu}\right) &=&
\frac{\alpha_{n,s}^{q+1}\,\mu^{(n-2s)(q+1)}}{(\mu^2+|y|^2)^{\frac{(n-2s)(q+1)}{2}}}\\
&\le & \frac{\alpha_{n,s}^{q+1}\,\mu^{(n-2s)(q+1)}}{|y|^{(n-2s)(q+1)}}\\
&\le & \frac{2^{(n-2s)(q+1)}\,\alpha_{n,s}^{q+1}\,\mu^{(n-2s)(q+1)}}{|\xi|^{(n-2s)(q+1)}}.
\end{eqnarray*}
Using this, \eqref{h2} and the fact that $h$ is continuous
into \eqref{conto2}, we have that
$$ |\Gamma(\mu,\xi)| \le C\, \mu^{\gamma_s}\,\frac{1}{|\xi|^{(n-2s)(q+1)}}\,\|h\|_{L^1(\R^n)}, $$
for some constant independent on $\mu$ and $\xi$.
Since $\mu\rightarrow\bar{\mu}\in(0,+\infty)$, this implies that
$$ \Gamma(\mu,\xi)\rightarrow 0 \ {\mbox{ as }}|\xi|\rightarrow +\infty,$$
thus concluding the proof of Lemma \ref{lemma2}.
\end{proof}

Finally we show the following:

\begin{lemma}\label{lemma3}
Assume that $\frac{2s}{n-2s}<q<p$.
Let $\Gamma$ be as in \eqref{gamma}. Suppose that there exists $\xi_0\in\R^n$
such that
$h(\xi_0)>0$ ($h(\xi_0)<0$ respectively). Then
$$ \lim_{\mu\rightarrow 0}\frac{\Gamma(\mu,\xi_0)}{\mu^{n-\gamma_s}}=A, $$
for some $A>0$ ($A<0$, respectively).
\end{lemma}

\begin{proof}
We prove the lemma only in the case $h(\xi_0)>0$, since the other case is analogous.
We notice that, by using the change of variable $y=(x-\xi)/\mu$,
we can rewrite $\Gamma$ as
\begin{equation*}
\Gamma(\mu,\xi)=\frac{\mu^{n-\gamma_s}}{q+1}\,\int_{\R^n}h(\mu y+\xi)\,z_0^{q+1}(y)\,dy.
\end{equation*}
Then we obtain
\begin{equation}\label{gamma2gamma2}
\frac{\Gamma(\mu,\xi_0)}{\mu^{n-\gamma_s}} =\frac{1}{q+1}\,
\int_{\R^n} h(\mu y+\xi_0)\,z_0^{q+1}(y)\,dy. \end{equation}
Now, since~$\frac{2s}{n-2s}<q<p$, we have that $z_0$ defined in \eqref{zetazero} satisfies
\begin{equation}\label{m00}
z_0^{q+1}\in L^1(\R^n).
\end{equation}
We observe that
$$ h(\mu y+\xi_0)\,z_0^{q+1}(y)\rightarrow h(\xi_0)\,z_0^{q+1}(y) \ {\mbox{ as }}\mu\rightarrow 0.$$
Moreover, thanks to \eqref{h2}, the fact that $h$ is continuous
and \eqref{m00}, we have that
$$ h(\mu y+\xi_0)\,z_0^{q+1}(y)\le \|h\|_{L^\infty(\R^n)}\, z_0^{q+1}(y)\in L^1(\R^n),$$
and so from the Dominated Convergence Theorem, we get
$$ \frac{\Gamma(\mu,\xi_0)}{\mu^{n-\gamma_s}}\rightarrow \frac{h(\xi_0)}{q+1}\,
\int_{\R^n}z_0^{q+1}(y)\,dy \ {\mbox{ as }}\mu\rightarrow 0,$$
as desired.
Notice indeed that
$$ A:=\frac{h(\xi_0)}{q+1}\,\int_{\R^n}z_0^{q+1}(y)\,dy $$
is strictly positive and bounded.\end{proof}

We also need a variation of Lemma \ref{lemma3}
to deal with the case in which $0<q\le\frac{2s}{n-2s}$.
In this case, recalling the alternative in \eqref{ALT:1}--\eqref{ALT:2},
we take the additional assumption that $h\ge0$.

\begin{lemma}\label{lemma3BIS}
Assume that $0<q\le\frac{2s}{n-2s}$ and $h\ge0$.
Let $\Gamma$ be as in \eqref{gamma}. Suppose that there exists $\xi_0\in\R^n$ such that
$h(\xi_0)>0$. Then
$$ \lim_{\mu\rightarrow 0}\frac{\Gamma(\mu,\xi_0)}{\mu^{n-\gamma_s}}=+\infty.$$
\end{lemma}

\begin{proof}
By \eqref{gamma2gamma2} and Fatou's Lemma,
$$ \liminf_{\mu\rightarrow 0}\int_{\R^n}h(\mu y+\xi_0)\,z_0^{q+1}(y)\,dy \ge
h(\xi_0)\,\int_{\R^n}z_0^{q+1}(y)\,dy=+\infty,$$
as desired.\end{proof}

\section{Proof of Theorems \ref{TH1} and \ref{TH1BIS}}\label{sec:proof}

Now we are ready to complete the proof of Theorems \ref{TH1} and \ref{TH1BIS}.
For this, we will use Lemma \ref{lemma3}
if alternative \eqref{ALT:1} holds true and 
Lemma \ref{lemma3BIS}
if alternative \eqref{ALT:2} holds true.
So we let $A$ to be as in Lemma \ref{lemma3} in the first case and $A:=+\infty$
in the second case. In this way,
thanks to \eqref{h3} and either Lemma \ref{lemma3} (if \eqref{ALT:1} is satisfied)
or Lemma \ref{lemma3BIS} (if \eqref{ALT:2} is satisfied), we see that
there exist $\mu_0>0$ as small as we want and $\xi_0\in\R^n$ such that
\begin{equation}\label{cond0}
\Gamma(\mu_0,\xi_0)\ge \frac{\mu_0^{n-\gamma_s}}{2}\,\min\{A,1\}=:B.
\end{equation}
Now, we use Lemma~\ref{lemma1} to say that if $\mu>0$ is sufficiently small, then
$$ \Gamma(\mu,\xi)<\frac{B}{2}  {\mbox{ for any }} \xi\in\R^n.$$
In particular, if~$\mu_1:=\mu_0/2$, then $\mu_1$ is small if so is $\mu_0$
and therefore we can write
\begin{equation}\label{cond1}
\Gamma(\mu_1,\xi)<\frac{B}{2} {\mbox{ for any }} \xi\in\R^n.
\end{equation}
Moreover, from Lemma~ \ref{lemma2} we deduce that there exists $R_*>0$
such that if $\mu+|\xi|>R_*$ we have
$$ \Gamma(\mu,\xi)<\frac{B}{2}.$$
In particular, we can take $\mu_2=R_2=R_*+\mu_0+|\xi_0|+1$ and we have that
\begin{equation}\label{cond2}
\Gamma(\mu,\xi)<\frac{B}{2} {\mbox{ if either }}\mu =\mu_2 {\mbox{ and }}|\xi|\le R_2
{\mbox{ or }}\mu\le\mu_2 {\mbox{ and }}|\xi|= R_2.
\end{equation}

Now we perform our choice of $R$, $\mu_1$ and $\mu_2$ in \eqref{critman}:
we take $\mu_1$ and $\mu_2$
such that \eqref{cond1} and \eqref{cond2} are satisfied,
and $R=R_2$.

Also, we set
$$ S:=\{\mu_1\le\mu\le\mu_2 {\mbox{ and }} |\xi|\le R\},$$
and we notice that $\Gamma$ admits a maximum in $S$, since $\Gamma$
is continuous and $S$ is a compact set. Moreover, thanks to~\eqref{cond1}
and~\eqref{cond2} we have that
\begin{equation}\label{cond4}
\Gamma(\mu,\xi)<\frac{B}{2} {\mbox{ if }}(\mu,\xi)\in\partial S.
\end{equation}
On the other hand,
$$ |\xi_0|<R_2 {\mbox{ and }} \mu_1<\mu_0<\mu_2, $$
which implies that $(\mu_0,\xi_0)\in S$.
Therefore, \eqref{cond0} and \eqref{cond4} imply that
the maximum of $\Gamma$ is achieved
at some point $(\mu_*,\xi_*)$ in the interior of $S$.

Now, we go back to the functional $G$,
and recalling \eqref{gamma} we obtain that $G$ admits
a maximum~$z_{\mu_*,\xi_*}$ in the critical manifold $Z_0$ defined in~\eqref{critman}.
Hence, we can apply Proposition~\ref{prop1} and we obtain the existence
of a critical point of $f_\epsilon$, that is a solution to~\eqref{problem},
given by
$$ u_{1,\epsilon}:= z_{\mu_*,\xi_*}+ w(\epsilon,z_{\mu_*,\xi_*}). $$
Also, $u_{1,\epsilon}$ is positive thanks to \eqref{est}.
This completes the proof of Theorem \ref{TH1}.

So we now focus on the proof of Theorem \ref{TH1BIS}. Notice that
in this case we are assuming that $\frac{2s}{n-2s}<q<p$
and so we are in the position of using Lemma \ref{lemma3}.
More precisely,
since $h$ changes sign, there exists~$\tilde{\xi}_0\in\R^n$
such that~$h(\tilde{\xi}_0)<0$, and so we can use Lemma~\ref{lemma3} to say that
$$ \Gamma(\tilde{\mu}_0,\tilde{\xi}_0)\le \frac{\tilde{\mu}_0^{n-\gamma_s}}{2}\,
\max\{A,-1\}, $$
for some $\tilde{\mu}_0>0$. Then we can
repeat all the above arguments (with suitable modifications)
to find a local minimum of~$\Gamma$, and so a a local minimum of~$G$.
Then, again from Proposition~\ref{prop1}
we obtain the existence of a second positive solution.
This concludes the proof of Theorem \ref{TH1BIS}.

\vfill

\end{document}